\documentclass[11pt, reqno]{amsart}
\usepackage{amsfonts,amssymb,latexsym,amsmath, amsxtra, nccmath, amsthm}
\usepackage{verbatim}
\usepackage{url}
\usepackage{enumitem}
\usepackage{tikz}
\usepackage{cancel}
\usepackage{bm}
\usepackage{csquotes}
\allowdisplaybreaks
\theoremstyle{plain}
\newtheorem{theorem}{Theorem}[section]
\newtheorem*{theorem*}{Theorem}
\newtheorem{corollary}[theorem]{Corollary}

\newtheorem{lemma}[theorem]{Lemma}
\newtheorem{proposition}[theorem]{Proposition}
\newtheorem*{conjecture*}{Conjecture}

\theoremstyle{definition}

\newtheorem*{rem}{Remark}
\theoremstyle{remark}

\newtheorem*{remark*}{Remark}
\newtheorem*{remarks*}{Remarks}

\numberwithin{equation}{section}

\newcommand{\R}{\mathbb R}
\newcommand{\N}{\mathbb N}
\newcommand{\Z}{\mathbb Z}
\newcommand{\C}{\mathbb C}

\newcommand{\cS}{\mathcal{S}}

\newcommand{\erf}{\operatorname{erf}}
\newcommand{\zz}{\mathfrak{z}}

\def\({\left(}
\def\){\right)}

\def\lp{\left(}
\def\rp{\right)}

\def\b{\beta}
\def\d{\delta}

\def\w{\omega}

\def\e{\varepsilon}

\def\g{\gamma}
\def\t{\tau}

\def\OO{\lambda}
\def\GG{\gamma}

\renewcommand{\pmod}[1]{\ \left( \mathrm{mod} \, #1 \right)}
\newcommand{\Pmod}[1]{\ ( \mathrm{mod} \, #1 )}

\renewcommand{\pmod}[1]{\,\,({\rm mod}\,\,{#1})}

\newcommand{\re}[1]{\operatorname{Re}\(#1\)}
\newcommand{\im}[1]{\operatorname{Im}\(#1\)}

\newcommand{\lcm}{\mathrm{lcm}}

\newcommand{\sgn}{\operatorname{sgn}}

\newcommand{\SL}{\operatorname{SL}}
\newcommand{\Log}{\operatorname{Log}}
\renewcommand{\H}{\mathbb{H}}
\newcommand{\fbm}{}
\newcommand{\kzer}{K}

\def\lp{\left(}
\def\rp{\right)}

\title[Representations of integers as sums of four polygonal numbers]{Representations of integers as sums of four polygonal numbers and partial theta functions}
\author{Kathrin Bringmann}
\address{Department of Mathematics and Computer Science\\Division of Mathematics\\University of
Cologne\\ Weyertal 86-90 \\ 50931 Cologne \\Germany}
\email{kbringma@math.uni-koeln.de}
\author{Min-Joo Jang}
\address{Department of Mathematics\\ University of Hong Kong\\ Pokfulam, Hong Kong}
\email{mjjang@hku.hk}
\email{bkane@hku.hk}
\email{atch1024@gmail.com}
\author{Ben Kane}
\author{Cheuk Hin Alvin Tse}
\date{\today}
\thanks{The research of the first author was supported by the Deutsche Forschungsgemeinschaft (DFG) Grant No. BR 4082/5-1. The research of the third author was supported by grants from the Research Grants Council of the Hong Kong SAR, China (project numbers HKU 17301317, 17303618, 17314122, and 17305923). }

\begin{document}
\begin{abstract}
In this paper, we consider representations of integers as sums of at most four distinct polygonal numbers with a prescribed number of repeats of each distinct polygonal number. We compare such representations with classical polygonal numbers, and those representations with generalized polygonal numbers. Our main result is that representations with classical polygonal numbers are equidistributed in the sense that the number of representations in the nonnegative quadrant in four-dimensional space is asymptotically $\frac{1}{16}$ of the representations in the entire space.
\rm
\end{abstract}
	\maketitle

\section{Introduction and statement of results}

The study of representations of integers as sums of polygonal numbers has a long and storied history. For $m\in\N_{\geq 3}$ and $\ell\in\N_0$, let $p_m(\ell)$ be the \begin{it}$\ell$-th $m$-gonal number\end{it}
\begin{equation*}
p_m(\ell):=\frac12 (m-2)\ell^2-\frac12 (m-4)\ell,
\end{equation*}
which counts the number of points in a regular $m$-gon with side lengths $\ell$. Fermat famously conjectured in 1638 that every positive integer may be written as the sum of at most $m$ $m$-gonal numbers, or equivalently that for every $n\in\N_0$
\begin{equation}\label{eqn:Fermatsum}
\sum_{1 \leq j \leq m} p_m(\ell_j) = n
\end{equation}
is solvable. Lagrange proved the four-squares theorem in 1770, resolving the case $m=4$ of Fermat's conjecture. The case $m=3$ of triangular numbers was solved by Gauss in 1796 and is sometimes called the Eureka Theorem because Gauss famously marked in his diary ``EYPHKA! num=$\triangle+\triangle+\triangle$''. Cauchy \cite{Cauchy} finally completed the full proof of the conjecture in 1813, and Nathanson \cite{Nathanson} shortened Cauchy's proof in 1987; he also provided some additional history.

More generally, for\footnote{Throughout we write vectors in bold letters.} $\fbm\bm{\alpha }\in \N^{\kappa}$  and $n\in\N$ one may consider Diophantine equations of the type
\begin{equation}\label{eqn:genpolysum}
\sum_{1 \leq j \leq \kappa} \alpha_j p_m\left(\ell_j\right)=n.
\end{equation}
It is natural to ask for a classification of those $n\in\N$ for which \eqref{eqn:genpolysum} is solvable with $\fbm\bm{\ell}\in \N_0^{\kappa}$.
The case $m=4$ is well-understood: by applying the theory of modular forms (see \cite[Proposition 11]{Zagier123}), for $m=4$ one not only knows the existence of a solution to \eqref{eqn:Fermatsum} but has a precise formula for the number of such solutions. Namely,
Jacobi showed in 1834 (see e.g. \cite[p. 119]{Williams}) that
\begin{equation}\label{eqn:sum4squares}
	\#\left\{\fbm\bm{\ell}\in\Z^4: \sum_{1\le j\le4} \ell_j^2=n\right\} = 8\sum_{\substack{d\mid n\\4\nmid d}} d.
\end{equation}
Although formulas like \eqref{eqn:sum4squares} are rare, they are often ``almost true'' in the sense that the number of solutions to equations like \eqref{eqn:genpolysum} with $\fbm\bm{\ell}\in\Z^{\kappa}$ may be written in the shape of \eqref{eqn:sum4squares} up to an error term. For example, in the case $\fbm\bm{\alpha}=\fbm\bm{1}$ with arbitrary even $\kappa$ and $m=4$, Ramanujan stated \cite[(146)]{Ramanujan} a formula for the number of solutions to \eqref{eqn:genpolysum} which was later proven by Mordell \cite{Mordell}.
Set
\[
r_{2k}(n):=\#\left\{\fbm\bm{\ell}\in\Z^{2k}: \sum_{1\leq j \leq 2k}\ell_j^2=n\right\}
\]
and suppose for simplicity that $k\geq 10$ is even. Ramanujan's claim \cite[(146)]{Ramanujan} together with \cite[(143)]{Ramanujan} implies that there exists $\delta>0$ such that for $n\in\N$
\begin{equation}\label{eqn:r2k}
r_{2k}(n)=\frac{2k(-1)^{n+1}}{\left(2^{k}-1\right)B_{k}} \sum_{d\mid n} (-1)^{d+\frac{k}{2}\frac{n}{d}} d^{k-1}+O\left(n^{k-1-\delta}\right),
\end{equation}
where $B_{k}$ is the $k$-th Bernoulli number. More generally, Kloosterman \cite[(I.3I)]{Kloosterman} applied the Circle Method to show formulas resembling \eqref{eqn:r2k} (where the main term is the singular series from the Circle Method) in the case $m=4$ and $\kappa=4$ of \eqref{eqn:genpolysum}. 

The goal of this paper is to obtain formulas resembling \eqref{eqn:r2k} for the number of solutions
\begin{equation*}
r_{m,\fbm\bm{\alpha}}(n):=\#\left\{\fbm\bm{\ell}\in \N_0^{\kappa}: \sum_{1\leq j \leq \kappa}\alpha_j p_m\left(\ell_j\right)=n\right\}.
\end{equation*}
In \eqref{eqn:sum4squares} and \eqref{eqn:r2k}, we count solutions with $\ell_j\in\Z$, while in this paper we restrict to solutions with $\ell_j\in\N_0$. The reason for this restriction is the connection with regular polygons. Although the formula defining $p_m(\ell_j)$ is still well-defined for $\ell_j\in\Z$, their interpretation as the number of points in a regular $m$-gon with side lengths $\ell_j$ is lost if $\ell_j<0$ because side-lengths cannot be negative. For $m\in\{3,4\}$, the restriction $\ell_j\in\N_0$ does not lead to a fundamentally different question than taking $\ell_j\in\Z$. Indeed, using that $p_3(-\ell-1)=p_3(\ell)$, we obtain for $m=3$ a bijection between solutions with $\ell_j\geq 0$ and those with $\ell_j<0$. Similarly, since $p_4(-\ell)=p_4(\ell)$, we have for $m=4$ a bijection between solutions with $\ell_j\geq 0$ and those with $\ell_j\leq 0$. The case $\ell_j=0$ is double-counted, but formulas for solutions with $\ell_j=0$ may be obtained by taking $\kappa\mapsto \kappa-1$ and removing $\alpha_j$ in \eqref{eqn:genpolysum}. Thus for $m\in\{3,4\}$, finding the number of solutions to \eqref{eqn:genpolysum} with $\ell_j\in\N_0$ is equivalent to finding the number of solutions with $\ell_j\in\Z$, and we hence assume $m\geq 5$ throughout.  To the best of our knowledge, in this case formulas like \eqref{eqn:r2k} for the number of solutions to \eqref{eqn:genpolysum} if $\fbm\bm{\ell}\in\N_0^{\kappa}$ are not known. However, standard techniques yield formulas of this type for $\fbm\bm{\ell}\in\Z^{\kappa}$. Completing the square in \eqref{eqn:genpolysum}, solutions to \eqref{eqn:genpolysum} are in one-to-one correspondence with solutions to certain sums of squares with fixed congruence conditions. Using this relationship, one finds that studying
\[
r_{m,\fbm\bm{\alpha}}^*(n):=\#\left\{\fbm\bm{x}\in\Z^{\kappa}: \sum_{1\leq j \leq \kappa}\alpha_jp_m(x_j) = n\right\}
\]
is equivalent to evaluating $s_{r,M,\bm\alpha}^*(An+B)$ (for some appropriate $A$, $B$, $r$, and $M$), where
\begin{equation*}\label{eqn:s*def}
s_{r,M,\fbm\bm{\alpha}}^*(n):=\#\left\{\fbm\bm{x}\in\Z^{\kappa}: \sum_{1\leq j\leq \kappa}\alpha_jx_j^2 = n,\ x_j\equiv r\pmod{M}\right\}.
\end{equation*}
The generating function ($q:=e^{2\pi i \tau}$ with $\tau\in\H:=\{\tau\in\C: \im{\tau}>0\}$)
\[
\Theta_{r,M,\fbm\bm{\alpha}}^*(\tau):=\sum_{n\geq 0}s_{r,M,\fbm\bm{\alpha}}^*(n) q^{\frac{n}{M}}
\]
is a modular form of weight $\frac{\kappa}{2}$ for some congruence subgroup (see, e.g., \cite[Proposition 2.1]{Shimura}). Using the theory of modular forms, formulas like \eqref{eqn:r2k} may be obtained by splitting $\Theta_{r,M,\fbm\bm{\alpha}}^*$ into an Eisenstein series and a cusp form and using a result of Deligne \cite{Deligne} to bound the Fourier coefficients of the cusp form as an error term. Explicit bounds for the Fourier coefficients of the cusp forms played an important role in a number of related problems, such as the theorem of Bhargava and Hanke \cite{BhargavaHanke} saying that positive definite integer-valued quadratic forms represent every positive integer if and only if they represent every positive integer up to $290$ and the conditional proof of Rouse \cite[Theorems 2 and 7]{Rouse} stating that every such quadratic form represents every odd integer if and only if it represents every positive integer up to $451$. As noted above, although these techniques yield formulas like \eqref{eqn:r2k} for $r_{m,\fbm\bm{\alpha}}^*(n)$ due to the connection with modular forms, one loses the interpretation for $p_{m}(\ell_j)$ in terms of regular $m$-gons. The aim of this paper is to link the study of $r_{m,\fbm\bm{\alpha}}(n)$ to modular forms while simultaneously preserving the connection with regular $m$-gons by restricting to $\ell_j\in\N_0$. However, the restriction of $\ell_j$ to $\N_0$ breaks an important symmetry and as a result the generating function for $r_{m,\fbm\bm{\alpha}}(n)$ is unfortunately not a modular form. Hence the standard techniques described above cannot be employed to obtain a formula for $r_{m,\fbm\bm{\alpha}}(n)$. Indeed, in his last letter to Hardy in 1920, Ramanujan commented that ``unlike the `False' theta functions'', the mock theta functions that he discovered ``enter into mathematics as beautifully as the ordinary theta functions''. However, contrary to Ramanujan's claims about the false theta functions, recent work by Nazaroglu and the first author \cite{BringmannNazaroglu} shows that the generating function has some modular properties and in particular can be \enquote{completed} to a function transforming like a modular form. This gives that the generating function has some explicit ``obstruction to modularity''. The investigation of this obstruction to modularity plays a fundamental role in this paper and causes most of the technical difficulties.

Given the results in \cite{BringmannNazaroglu}, one approach to obtaining formulas like \eqref{eqn:r2k} would be to establish structure theorems or an expansion of results on modular forms to extend to functions with this type of obstruction to modularity. In this paper, we instead link the $r_{m, \fbm\bm{\alpha}}(n)$ and $r_{m,\fbm\bm{\alpha}}^*(n)$, showing that they are essentially equal up to an error term.  As above, by completing the square, one finds that this is equivalent to relating $s_{r,M,\fbm\bm{\alpha}}^*(An+B)$ to $s_{r,M,\bm\alpha,C}(An+B)$ (for some  $A,B,C$), where
\[
s_{r,M,\fbm\bm{\alpha},C}(n):=\#\left\{\fbm\bm{x}\in\Z^{\kappa}: \sum_{1\leq j\leq \kappa}\alpha_jx_j^2 = n,\ x_j\equiv r\pmod{M},\ x_j\geq C\right\}.
\]
If $C=1$ (i.e., if $\fbm\bm{x}\in\N^\kappa$), then we omit it in the notation.
Heuristically, one would expect that solutions with $\varepsilon_j x_j>0$ are equally distributed independent of the choice of $\varepsilon_j\in \{\pm 1\}$. Our main theorem shows that this is indeed the case.
\begin{theorem}\label{thm:rNrZ}
Let $\fbm\bm{\alpha}\in \N^4$ and $r,M\in\N$.
\noindent

\noindent
\begin{enumerate}[leftmargin=*,label=\rm(\arabic*)]
\item  We have
\begin{equation*}\label{eqn:sNsZ}
s_{r,M,\fbm\bm{\alpha}}(n)=\frac{1}{16}s_{r,M,\fbm\bm{\alpha}}^*(n)+O\left(n^{\frac{15}{16}+\varepsilon}\right).
\end{equation*}
\item
For $m>4$ we have
\begin{equation*}\label{eqn:rNrZ}
r_{m,\fbm\bm{\alpha}}(n) = \frac{1}{16} r_{m,\fbm\bm{\alpha}}^*(n) + O\left(n^{\frac{15}{16}+\varepsilon}\right).
\end{equation*}
\end{enumerate}
\end{theorem}
\noindent
\begin{rem}
The main term of $s_{r,M,\fbm\bm{\alpha}}^*(n)$ comes from the Eisenstein component of $\Theta_{r,M,\fbm\bm{\alpha}}^*$. The computation of the corresponding Eisenstein series appears throughout the literature in a variety of different shapes. In one direction,
 Kloosterman \cite{Kloosterman} computed this component as the singular series coming from the Circle Method. On the other hand, the corresponding Eisenstein series appears in the  work of Siegel \cite{Siegel1,Siegel2} and follow-up work of Weil \cite{Weil}, van der Blij \cite{vanderBlij}, and Shimura \cite{ShimuraCongruence} in two different forms. Firstly,  the Eisenstein series may be realized as a certain weighted average of the solutions over the genus of the given sum of squares with congruence conditions. Secondly, Siegel computed its Fourier coefficients as certain $p$-adic limits. Finally, since the space of modular forms of a given weight and congruence subgroup is a finite-dimensional vector space, one may explicitly construct a basis and determine the Eisenstein series component using Linear Algebra.
\end{rem}
As noted above, combining Theorem \ref{thm:rNrZ} with known techniques from the theory of modular forms yields formulas resembling \eqref{eqn:r2k}. As a first corollary, we obtain such a formula for the number of representations of $n$ as a sum of four hexagonal numbers;  the main term is given in terms of the \emph{sum of divisors function} $\sigma(n):=\sum_{d\mid n} d$.
\begin{corollary}\label{cor:hexagonal}
We have
\[
r_{6,(1,1,1,1)}(n)=\frac{1}{16}\sigma(2n+1)+O\left(n^{\frac{15}{16}+\varepsilon}\right).
\]
\end{corollary}
\begin{rem}
Since $\sigma(2n+1)\geq 2n+1$, Corollary \ref{cor:hexagonal} implies that $r_{6,\fbm\bm{\alpha}}(n)>0$ for $n$ sufficiently large. Guy \cite{Guy} proposed a study of the numbers which are not the sum of four polygonal numbers.  Moreover, Corollary \ref{cor:hexagonal} implies that the number of such solutions is $\gg n$.
\end{rem}
Another example is given by sums of five hexagonal numbers where the last hexagonal number is repeated at least twice.
To state the result, let $(\frac{\cdot}{\cdot})$ be the generalized Legendre symbol.
\begin{corollary}\label{cor:hexagonal2}
For $\fbm\bm{\alpha}={\(1,1,1,2\)}$ and $m=6$, we have
\[
r_{6,\fbm\bm{\alpha}}(n)=-\frac{1}{64}\sum_{d\mid (8n+5)} \left(\frac{8}{d}\right)d +O\left(n^{\frac{15}{16}+\varepsilon}\right).
\]
In particular, for $n$ sufficiently large
\[
r_{6,\fbm\bm{\alpha}}(n)>0.
\]
\end{corollary}
The proofs of Corollaries \ref{cor:hexagonal} and \ref{cor:hexagonal2} rely on formulas of Cho \cite{Cho} which use the fact that $\Theta_{-1,4,\bm{\alpha}}^*$ is an Eisenstein series in the cases
$\bm{\alpha}=(1,1,1,1)$ and $\bm{\alpha}=(1,1,1,2)$; indeed, as pointed out by Cho in \cite[Examples 3.3 and 3.4]{Cho}, the space of modular forms containing them is spanned by Eisenstein series in these cases. However, to obtain similar corollaries from Theorem \ref{thm:rNrZ}, we do not require the corresponding theta function to be an Eisenstein series. In order to exhibit how to use Theorem \ref{thm:rNrZ}, we give one such example.
\begin{corollary}\label{cor:pentagonal}
For $\fbm\bm{\alpha}={\(1,1,1,1\)}$ and $m=5$, we have
\[
r_{5,(1,1,1,1)}(n)= \frac{1}{24}\sigma(6n+1) + O\left(n^{\frac{15}{16}+\varepsilon}\right).
\]
\end{corollary}

The paper is organized as follows. In Section \ref{sec:setup}, we connect sums of squares and polygonal numbers, introduce partial theta functions, and relate them to theta functions and false theta functions. In Section \ref{sec:Farey}, we recall some facts about Farey fractions that are used for the Circle Method. In Section \ref{sec:modular}, we give modular transformation properties of the theta functions and the false theta functions in a shape that is useful for our application of the Circle Method. Section \ref{sec:IntBound} is devoted to studying the obstruction to modularity of the false theta functions and bounding them in a suitable way to use in the Circle Method. In Section \ref{sec:CircleMethod}, we prove a modified version of Kloosterman's fundamental lemma \cite[Lemma 6]{Kloosterman} and apply the Circle Method to prove Theorem \ref{thm:rNrZ}. Finally, we prove Corollaries \ref{cor:hexagonal}, \ref{cor:hexagonal2}, and \ref{cor:pentagonal}
in Section \ref{sec:Corollaries} to demonstrate how to apply Theorem \ref{thm:rNrZ} to obtain identities resembling \eqref{eqn:r2k}.

\section{Sums of squares with congruence conditions and polygonal numbers} \label{sec:setup}
In this section, we relate sums of polygonal numbers and sums of squares and give a relationship between $s_{r,M,\fbm\bm{\alpha}}$ and $s_{r,M,\fbm\bm{\alpha}}^*$.
Without loss of generality, we pick the ordering $\alpha_j\geq \alpha_{j+1}$ for $j\in\{1,2,3\}$
in \eqref{eqn:genpolysum}. As noted in the introduction, we investigate sums of polygonal numbers via a connection with sums of squares satisfying certain congruence conditions.
Writing
\begin{equation}\label{pml}
p_m(\ell)=\frac12 (m-2)\left(\ell-\frac{m-4}{2(m-2)}\right)^2-\frac{(m-4)^2}{8(m-2)},
\end{equation}
one sees directly that
\[
r_{m,\fbm\bm{\alpha}}(n) = s_{-(m-4),2(m-2),\fbm\bm{\alpha},-(m-4)}\left(8(m-2)n+\sum_{1\leq j \leq 4} \alpha_j(m-4)^2\right).
\]

Using \eqref{pml}, we have the generating function
\begin{equation*}
\sum_{n\geq 0}r_{m, \fbm\bm{\alpha}}(n)q^n=\sum_{\fbm\bm{\ell}\in\N_0^4}q^{\sum_{j=1}^4\alpha_j p_m\left(\ell_j\right)
}=q^{-\sum_{j=1}^4\alpha_j\frac{(m-4)^2}{8(m-2)}}\prod_{j=1}^4 \sum_{\ell\ge0}q^{\alpha_j\frac{m-2}{2}\left(\ell-\frac{m-4}{2(m-2)}\right)^2}.
\end{equation*}

We restrict our investigation of solutions to \eqref{eqn:genpolysum} to the case $d=4$ and $\fbm\bm{\ell}\in\N_0^4$. We claim that most of the solutions to \eqref{eqn:genpolysum} come from solutions with $\ell_j\neq 0$, i.e., sums of precisely four polygonal numbers instead of at most four polygonal numbers. Indeed, the subset of solutions with one of the $\ell_j$ equal to zero solves a lower-dimensional equation of the same type. Defining $r_{m,\fbm\bm{\alpha}}^{+}(n)$ via the generating function
\begin{equation}\label{eqn:r+def}
\sum_{n\geq 0}r_{m, \fbm\bm{\alpha}}^{+}(n)q^n:=\sum_{\fbm\bm{\ell}\in\N^4}q^{\sum_{j=1}^4\alpha_j p_m\left(\ell_j\right)}=q^{-\sum_{j=1}^4\alpha_j\frac{(m-4)^2}{8(m-2)}}\prod_{j=1}^4 \sum_{\ell\ge1}q^{\alpha_j\frac{m-2}{2}\left(\ell-\frac{m-4}{2(m-2)}\right)^2}
\end{equation}
and using the bound from \cite[Lemma 4.1(a)]{Blomer} on the number of integer solutions in three variables, a direct calculation shows the following.
\begin{lemma}\label{lem:rexactly}
	For $\fbm\bm{\alpha}\in\N^4$, we have
	\[
		r_{m,\fbm\bm{\alpha}}(n)=r_{m,\fbm\bm{\alpha}}^{+}(n)+O\left(n^{\frac{1}{2}+\varepsilon}\right).
	\]
\end{lemma}

Define the partial theta function
\[
\Theta_{r,M,\fbm\bm{\alpha}}^+(\tau):=\sum_{n\geq 0}s_{r,M,\fbm\bm{\alpha}}(n) q^{\frac{n}{M}},
\]
which is closely related to the generating function of $r_{m,\fbm\bm{\alpha}}^+(n)$ by \eqref{eqn:r+def}.
\begin{lemma}\label{lem:r+s+rel}
For $m\geq 5$ and $\fbm\bm{\alpha}\in \N^4$, we have
\[
\sum_{n\geq 0} r_{m,\fbm\bm{\alpha}}^+(n) q^n =q^{-\sum_{j=1}^4\alpha_j\frac{(m-4)^2}{8(m-2)}}\Theta_{m,2(m-2),\fbm\bm{\alpha}}^+\left(\frac{\tau}{4}\right).
\]
\end{lemma}

By Lemma \ref{lem:r+s+rel} and Lemma \ref{lem:rexactly}, to prove Theorem \ref{thm:rNrZ} it suffices to approximate the Fourier coefficients of $\Theta_{r,M,\fbm\bm{\alpha}}^+(\tau)$. The partial theta functions get their name from the fact that while theta functions count the number of lattice points of a given distance from the origin, partial theta functions only those lattice points in a given subset, usually formed by splitting the space by a hyperplane. Another closely-related object is the false theta functions, which count the points on the lattice with a weighting sign-factor. In our case, the partial theta functions $\Theta_{r,M,\fbm\bm{\alpha}}^+$ are closely related to the usual (unary) theta functions $\vartheta(r,M;\tau)$ and false theta functions $F_{r,M}(\tau)$, defined for $-M< r \leq M$, $M\in \N$ by (using the convention $\sgn(0):=0$)
\begin{equation}\label{eqn:thetaFdef}
F_{r,M}(\tau):=\sum_{\nu\equiv r\pmod{2M}}\sgn(\nu)q^{\frac{\nu^2}{4M}},\qquad \vartheta(r,M; \tau):=\sum_{\nu\equiv r\pmod{M}}q^{\frac{\nu^2}{2M}}.
\end{equation}
A direct calculation shows the following.
\begin{lemma}\label{lem:ThetaFalse}
	For $M\in\N$, $-M<r\leq M$, and $\fbm\bm{\alpha}\in\N^4$, we have
	\[
		\Theta_{r,2M,\fbm\bm{\alpha}}^+(\tau)=\frac{1}{16}\sum_{J\subseteq\{1,2,3,4\}} \prod_{j\in J}\vartheta\left(r, 2M; 2\alpha_{j}\tau\right)\prod_{\ell\in \{1,2,3,4\}\setminus J} F_{r,M}\left(2\alpha_{\ell}\tau\right).
	\]
\end{lemma}

By Lemmas \ref{lem:r+s+rel} and \ref{lem:ThetaFalse}, for $J\subseteq\{1,2,3,4\}$ it is natural to define 
\[
F_{r,M,\fbm\bm{\alpha},J}(\tau):=q^{-\frac{r^2}{2M}\sum_{j=1}^4 \alpha_j}\prod_{j\in J}\vartheta\left(r, 2M; 2\alpha_{j}\tau\right)\prod_{\ell\notin J} F_{r,M}\left(2\alpha_{\ell}\tau\right),
\]
where hereafter $\ell \notin J$ means $\ell\in \{1,2,3,4\}\setminus J$.
Then for each $J\subseteq \{1,\dots,4\}$ we set 
\begin{equation*}
F_{r,M,\fbm\bm{\alpha},J}(\tau)=: \sum_{n\geq 0} c_{r,M,\fbm\bm{\alpha},J}(n)q^n.
\end{equation*}
If $J=\{1,2,3,4\}$, then we omit $J$ in the notation. A straightforward calculation yields the following.
\begin{lemma}\label{lem:mainterm}
\noindent

\noindent
\begin{enumerate}[leftmargin=*,label=\rm(\arabic*)]
\item
For $M\in\N$, $-M<r\leq M$, and $\fbm\bm{\alpha}\in \N^4$, we have
\[
F_{r,M,\fbm\bm{\alpha}}(\tau)= q^{-\frac{r^2}{2M}\sum_{j=1}^4 \alpha_j}\Theta_{r,2M,\fbm\bm{\alpha}}^*(\tau).
\]
In particular, for every $n\in\N_0$
\[
c_{r,M,\fbm\bm{\alpha}}(n)=s_{r,2M,\fbm\bm{\alpha}}^*\left(2Mn+r^2\sum_{1\leq j \leq 4}\alpha_j\right).
\]
\item
For $m\geq 5$ and $\fbm\bm{\alpha}\in\N^4$, we have
\[
c_{m,m-2,\fbm\bm{\alpha}}\left(4\left(n-\sum_{1\leq j \leq 4} \alpha_j\right)\right)=r_{m,\fbm\bm{\alpha}}^*(n).
\]
\end{enumerate}
\end{lemma}

\section{Basic facts on Farey fractions}\label{sec:Farey}

The {\it Farey sequence of order $N$ $\in \N$} is the sequence of reduced fractions in $[0,1)$ whose denominators do not exceed $N$. If $\frac{h}{k}$, $\frac{h_1}{k_1}$ are adjacent elements in the Farey sequence then their {\it mediant} is $\frac{h+h_1}{k+k_1}$.
When computing mediants below, we consider $\frac{N-1}{N}$ to be adjacent to $\frac{0}{1}$ and take the mediant between $\frac{N-1}{N}$ and $\frac{1}{1}$.
 The Farey sequence of order $N$ is then iteratively defined by placing the mediant between two adjacent Farey fractions of order $N-1$ if the denominator of the mediant in reduced terms is at most $N$. We see that two Farey fractions $\frac{h_1}{k_1}<\frac{h}{k}$ of order $N$ are adjacent if and only if the mediant in reduced terms has denominator larger than $N$. This implies that
\begin{equation}\label{eqn:adjacent}
h k_1-h_1k=1.
\end{equation}
The converse is also true: if $hk_1-h_1k=1$, then $\frac{h}{k}$ and $\frac{h_1}{k_1}$ are adjacent Farey fractions of order $\max\{k,k_1\}$. For three adjacent Farey fractions $\frac{h_1}{k_1}<\frac{h}{k}<\frac{h_2}{k_2}$, we set for $j\in\{1,2\}$ (note that $k_j$ depends on $h$)
\begin{equation}\label{eqn:rhohdef}
\varrho_{k,j}(h):=k+k_j-N.
\end{equation}
Since the mediant between adjacent terms has denominator larger than $N$ and $k_j\leq N$, we have
\begin{equation}\label{eqn:varrhojbnd}
1\leq \varrho_{k,j}(h)\leq k.
\end{equation}
The following lemma is straightforward to prove.
\begin{lemma}\label{lem:AdjacentNeighbours}
If $\frac{h}{k}<\frac{h_2}{k_2}$ are adjacent Farey fractions of order $N$, then $1-\frac{h_2}{k_2}<1-\frac{h}{k}$ are also adjacent Farey fractions of order $N$ and
\[
\varrho_{k,2}(h)=\varrho_{k,1}(k-h).
\]
\end{lemma}

For $n\in \N$, set $N:=\lfloor\sqrt{n}\rfloor$ and define arcs along the circle of radius $e^{-\frac{2\pi}{N^2}}$ through $e^{2\pi i \tau}$ with $\tau=\frac{h}{k}+\frac{iz}{k}\in \H$.
 Note that $\tau\in\H$ is equivalent to $\operatorname{Re}(z)>0$. Specifically, we choose
Farey fractions $\frac{h}{k}$ of order $N$ with $0\leqslant h<k\leqslant N$ and $\gcd(h,k)=1$
and set $z:=k(\frac{1}{N^2}-i\Phi)$ with $-\vartheta'_{h,k}\leqslant \Phi \leqslant \vartheta^{''}_{h,k}$. Here, for adjacent Farey fractions $\frac{h_1}{k_1}<\frac{h}{k}<\frac{h_2}{k_2}$ in the Farey sequence of order $N$, set
\begin{equation*}
\vartheta'_{h,k}:=\frac{1}{k(k+k_1)}, \quad \vartheta^{''}_{h,k}:=\frac{1}{k(k+k_2)}.
\end{equation*}
By \eqref{eqn:varrhojbnd}, we have
\begin{equation}\label{eqn:PhiBound}
|\Phi|\leq \max\left\{\vartheta_{h,k}^{'},\vartheta_{h,k}^{''}\right\}<\frac{1}{kN}.\quad \quad j\in\{1,2\}.
\end{equation}

\section{Modular Transformations}\label{sec:modular}

Kloosterman's version of the Circle Method \cite{Kloosterman} plays a fundamental role in the proof of Theorem \ref{thm:rNrZ}. We require the asymptotic behaviour %of $F_{r,M,\fbm\bm{\alpha},J}(\tau)$
towards $\tau=\frac{h}{k}$. Transformation properties relating the cusp $\frac{h}{k}$ to $i\infty$ thus play a pivotal role in determining the asymptotic growth near the cusp $\frac{h}{k}$.
To state these, for $a,b\in\Z$, $c\in\N$ define the \textit{Gauss sum}
\begin{equation*}
G(a,b;c):= \sum_{\ell\pmod{c}} e^{\frac{2\pi i}{c} \left(a\ell^2+b\ell\right)}.
\end{equation*}

\subsection{The theta functions}
We use the following modular transformation properties.
\begin{lemma}\label{lem:thetatrans}
	We have
	\begin{equation*}
	\vartheta\lp r,2M;2\alpha_j\lp\frac{h}{k}+\frac{iz}{k}\rp\rp
	=\frac{e^{\frac{\pi i \alpha_jhr^2}{Mk}}}{2\sqrt{Mk\alpha_jz}}\sum_{\nu\in\Z}e^{-\frac{\pi  \nu^2}{4Mk\alpha_jz}+\frac{\pi i r \nu}{Mk}} G(2M\alpha_jh,2r\alpha_jh+\nu;k).
	\end{equation*}
\end{lemma}
\begin{proof}
Writing $\nu=r+2M\alpha+2Mk\ell$ with $\alpha\pmod{k}$ and $\ell\in\Z$ in definition \eqref{eqn:thetaFdef}, we obtain
\begin{align*}
\vartheta\lp r,2M;2\alpha_j\lp\frac{h}{k}+\frac{iz}{k}\rp\rp
= \sum_{\alpha\pmod k} e^{
\frac{2\pi i \alpha_j h}{2Mk}
(r+2M\alpha)^2} \sum_{\ell\in\Z}
e^{
\frac{2\pi i \alpha_j}{2Mk}
(r+2M\alpha+2Mk\ell)^2 iz}.
\end{align*}
Using the modular inversion formula (see \cite[(2.4)]{Shimura})
\begin{equation*}
\vartheta\lp r,M;-\frac 1 \tau\rp = M^{-\frac 12} \sqrt{-i\tau} \sum_{k\pmod{M}} e^{\frac{2\pi i rk}{M}}
\vartheta(k,M;\tau)
\end{equation*}
on the inner sum, the claim easily follows.
\end{proof}

\subsection{The false theta functions}
We next establish analogous modular properties for the false theta functions.
For $\mu\in\Z\setminus\{0\}$ set
\begin{equation}\label{defineint}
\mathcal{I}(\mu,k;z)=\mathcal{I}_{M,\alpha_j}(\mu,k;z):=\lim_{\varepsilon\to 0^+}\int_{-\infty}^\infty \frac{e^{-\frac{\pi x^2}{4Mk\alpha_jz}}}{x-(1+i\varepsilon)\mu}dx.
\end{equation}
Throughout, we write $\sum^{*}_{\nu\geq 0}$ for the sum where the term $\nu=0$ is counted with a factor $\frac{1}{2}$
and  moreover abbreviate
\[
\sideset{}{^*}\sum_{\fbm\bm{\nu}\in \N_0^4}:= \prod_{j=1}^4\sideset{}{^*}\sum_{\nu_j\geq 0}.
\]
For $d\in\N$, set
\[
\mathcal{L}_{d}:=[1-d,-1]\cup[1,d].
\]
\begin{lemma}\label{lem:Fmodular}
	We have
\begin{multline*}
F_{r,M}\lp 2\alpha_j\lp \frac{h}{k}+\frac{iz}{k}\rp\rp=\frac{1}{2\sqrt{Mk\alpha_jz}}e^{\frac{\pi i \alpha_jhr^2}{Mk}}\sum_{\nu\in\Z}\operatorname{sgn}(\nu) e^{-\frac{\pi \nu^2}{4Mk\alpha_jz}+\frac{\pi i r \nu}{Mk}}G(2M\alpha_jh,2\alpha_jhr+\nu;k) \\
+ \frac{ie^{\frac{\pi i \alpha_jh r^2}{Mk}}}{2\sqrt{Mk\alpha_jz}\pi}
\sum_{\ell\in\mathcal{L}_{Mk}}
\sideset{}{^*}\sum_{\nu\geq 0}\sum_{\pm} e^{\frac{\pi i r \ell}{Mk}}G(2M\alpha_jh,2\alpha_jhr+\ell;k)\mathcal{I}(\ell\pm 2Mk\nu,k;z).
\end{multline*}
\end{lemma}

\begin{proof}
We have, writing $\nu=r+2M\alpha +2Mk\ell$ ($0\leq\alpha\leq k-1$, $\ell\in\Z$)
\begin{align*}
F_{r,M}\lp2\alpha_j\lp\frac{h}{k}+\frac{iz}{k}\rp\rp
&=\sum_{\alpha=0}^{k-1} e^{\frac{\pi i \alpha_jh}{Mk}(r+2M\alpha)^2} F_{r+2M\alpha,Mk}(2\alpha_j i z).
\end{align*}
Choosing the $+$-sign in  \cite[two displayed formulas after (4.5)]{BringmannNazaroglu} implies that
\begin{equation*}
F_{\beta,M}\lp-\frac 1\t\rp -\t^{\frac 12} \sum_{r=1}^{M-1}\psi_{\beta,r}\lp \begin{matrix}0& -1\\ 1 & 0\end{matrix}\rp F_{r,M}(\t)=
\sqrt{2M} \int_{0}^{-\frac 1\t +i\infty+\varepsilon} \frac{f_{\beta,M}(\zz)}{\sqrt{i\lp\zz+\frac 1\t\rp}}d\zz,
\end{equation*}
where
\begin{equation*}
f_{r,M}(\tau) :=  \frac{1}{2M}  \sum_{\nu\equiv r\pmod{2M}}\nu q^{\frac{\nu^2}{4M}},\qquad
\psi_{\beta,r}\lp\begin{matrix}0&-1\\ 1 &0\end{matrix}\rp:= e^{-\frac{3\pi i}{4}}\sqrt{\frac{2}{M}}\sin\lp\frac{\pi \beta r}{M}\rp.
\end{equation*}
Changing $\t\mapsto-\frac 1\t$, and using
\begin{equation*}
F_{0,M}(\tau)=F_{M,M}(\tau)=0, \qquad F_{2M-r,M}(\tau)=-F_{r,M}(\tau),
\end{equation*}
\begin{equation*}
\sum_{\beta \pmod{2M}} e^{\frac{2\pi i}{2M} (\ell+r)\beta} = \begin{cases}
0 & \text{ if } r \not\equiv -\ell \pmod{2M}, \\
2M & \text{ if } r \equiv -\ell \pmod{2M},
\end{cases}
\end{equation*}
we obtain, after a short calculation
\begin{multline}\label{eqn:Finverse}
F_{\ell,M}\lp-\frac{1}{\tau}\rp= e^{\frac{\pi i}{4}} \sqrt{-\frac{\tau}{2M}}\sum_{\beta\pmod{2M}} e^{\frac{2\pi i \ell \beta}{2M}} F_{\beta,M}(\tau)\\
+ e^{-\frac{3\pi i}{4}} \sqrt{-\tau}\sum_{\beta\pmod{2M}} e^{
\frac{2\pi i \ell \beta}{2M}} \int_{0}^{\t+i\infty+\varepsilon}\frac{f_{\beta,M}(\zz)}{\sqrt{i(\zz-\tau)}}d\zz.
\end{multline}
Thus
\begin{multline*}\label{rewriteF}
F_{r+2M\alpha,Mk}(2\alpha_j iz) = \frac{e^{\frac{\pi i}{4}}}{2\sqrt{Mk\alpha_j iz}}
\sum_{\beta\pmod{2Mk}} e^{\frac{2\pi i (r+2M\alpha)\beta}{2Mk}} F_{\beta,Mk}\lp \frac{i}{2\alpha_jz}\rp\\
+ \frac{e^{-\frac{3\pi i}{4}}}{\sqrt{2\alpha_j iz}}\sum_{\beta\pmod{2Mk}} e^{
\frac{2\pi i (r+2M\alpha)\beta}{2Mk}} \int_{0}^{\frac{i}{2\alpha_j z}+i\infty+\varepsilon}\frac{f_{\beta,Mk}(\zz)}{\sqrt{i\lp\zz-\frac{i}{2\alpha_jz}\rp}}d\zz.
\end{multline*}
 The first term can easily be rewritten, giving the first summand claimed in the lemma.

In the second term of \eqref{eqn:Finverse}, $f_{0,Mk}=0$ and for $\beta\neq 0$ and $\tau=\frac{i}{2\alpha_j z}$ we write the integral as
	\begin{align*}
	\frac{i}{2M}\lim_{\delta\rightarrow 0^+}
	\sum\limits_{\nu\equiv \beta \pmod{2M}} \nu e^{
\frac{\pi i \nu^2\t}{2M}} \int_{i\t+\delta}^{\infty-i\varepsilon} \frac{e^{-
\frac{\pi \nu^2\zz}{2M}
}}{\sqrt{-\zz}} d\zz.
	\end{align*}
We split up the integral in a way that allows $\delta=0$ to be directly plugged in termwise by Abel's Theorem. For this, we use \cite[displayed formula after (3.4)]{BringmannNazaroglu} to obtain that
\begin{equation*}
	\int_{i\tau+\delta}^{\infty-i\varepsilon}\frac{e^{-\frac{\pi \nu^2\mathfrak{z}}{2M}}}{\sqrt{-\mathfrak{z}}}d\mathfrak{z}=-\frac{i\sqrt{2M}}{\nu}\left(\sgn(\nu)+\erf\left(i \nu \sqrt{\frac{\pi}{2M} (-i\t-\delta)}\right)\right).
\end{equation*}
We split the error function as
\begin{equation}\label{eqn:errorsplit}
\lp \erf\lp i\nu \lp\sqrt{\frac{\pi}{2M}(-i\t-\delta)}\rp\rp   -\frac{ie^{\frac{\pi \nu^2}{2M} (-i\t-\delta)}}{\sqrt{2M}\pi \nu \sqrt{-i\t-\delta}} \rp + \frac{ie^{\frac{\pi \nu^2}{2M} (-i\t-\delta)}}{\sqrt{2M}\pi \nu \sqrt{-i\t-\delta}}.
\end{equation}
 Plugging in the asymptotic expansion of the error function towards $\infty$ for the error function, one finds that the series in $\nu$ of $\sgn(\nu)$ plus the first term of \eqref{eqn:errorsplit} converges absolutely for $\delta\geq 0$, and hence we may just take the limit $\delta\to 0^+$.
 For the second term, we need to compute
\begin{align*}
	\lim_{\delta\rightarrow 0^+} \frac{1}{\sqrt{-i\t-\delta}}
	\sum\limits_{\nu\equiv \beta \pmod{2M}}
	\frac{e^{-\frac{\pi \nu^2\delta}{2M}}}{\nu}  =\lim_{\delta\rightarrow 0^+} \frac{1}{\sqrt{-i\t-\delta}}\left(
\sum_{\nu\geq 1} \sum_{\pm}\frac{e^{-\frac{\pi}{2M} \left(\beta \pm 2M\nu\right)^2 \delta}}{\beta \pm 2M \nu}
+\frac{e^{-\frac{\pi \beta^2}{2M} \delta}}{\beta}
\right).
\end{align*}
Using the fact that ${\sum\limits_{\pm}} \frac{1}{\beta\pm 2M\nu} = \frac{2\beta}{\beta^2-4M^2\nu^2}$, the above series converges absolutely for $\delta \geq 0$ and hence by Abel's Theorem we have, for $\beta\neq 0$
\begin{multline*}
	\int_{0}^{\t+i\infty+\varepsilon} \frac{f_{\beta,M}(\zz)}{\sqrt{i(\zz-\t)}}d \zz\\
 = \frac{1}{\sqrt{2M}}
\sideset{}{^*}{\sum}_{\nu\geq 1}\sum_{\pm}
\left(\sgn(\beta\pm 2 M\nu)
	+\erf\left(i(\beta \pm 2M\nu)\sqrt{-\frac{\pi i\t}{2M}}\right)\right) e^{\frac{\pi i}{2M} (\beta \pm 2M\nu)^2\t}\\
+\frac{1}{\sqrt{2M}}
\left(\sgn(\beta)+
\erf
\left(i\beta \sqrt{-\frac{\pi i\t}{2M}}\right)
\right)
 e^{\frac{\pi i}{2M} \beta^2\t}.
\end{multline*}
We now use the following identity from \cite[(3.8)]{BringmannNazaroglu} ($s\in\R\setminus\{0\}$, $\operatorname{Re}(V)> 0$)
\begin{equation*}%\label{erroid}
\left(\sgn(s)+\erf\left(is\sqrt{\pi V}\right)\right)e^{-\pi s^2 V}
= -\frac{i}{\pi} \lim_{\varepsilon\to 0^+} \int_{-\infty}^{\infty} \frac{e^{-\pi V x^2}}{x-s(1+i\varepsilon)}dx,
\end{equation*}
 to obtain that
\begin{multline*}
\int_{0}^{\t+i\infty+\varepsilon} \frac{f_{\beta,M}(\zz)}{\sqrt{i(\zz-\t)}}d \zz = -\frac{i}{\sqrt{2M}\pi} \sum_{\nu\geq 1}\sum_\pm \lim_{\varepsilon\to 0^+}
\int_{-\infty}^{\infty} \frac{e^{\frac{\pi i \t x^2}{2M}}}{x-(1+i\varepsilon)(\beta \pm2M\nu)}dx\\
-\frac{i}{\sqrt{2M}\pi}
\lim_{\varepsilon\to 0^+}\int_{-\infty}^{\infty} \frac{e^{\frac{\pi i \t x^2}{2M}}}{x-(1+i\varepsilon)\beta}dx.
\end{multline*}
 From this the second claimed term in the lemma may directly be obtained.
\end{proof}

\section{Bounding $\mathcal{I}(\mu,k;z)$}\label{sec:IntBound}

\subsection{Rewriting $\mathcal{I}(\mu,k;z)$}

In the following lemma, we rewrite $\mathcal{I}(\mu,k;z)$. To state the lemma, set
\[
g(x):= e^{-\frac{\pi (x+\mu)^2}{4Mk\alpha_jz}},\qquad R_g(x):=\re{g(x)},\qquad I_g(x):=\im{g(x)}.
\]
\begin{lemma}\label{lem:inteval}
For every $\delta>0$ and $\mu\in\Z\setminus\{0\}$, we have
\begin{equation*}
\mathcal{I}(\mu,k;z)= \operatorname{sgn}(\mu)\pi i e^{-\frac{\pi \mu^2}{4Mk\alpha_j z}} +\int_{-\delta}^{\delta} \left(R_g'\left(y_{1,x}\right)+ iI_g'\!\left(y_{2,x}\right)\right)dx+\operatorname{sgn}(\mu)\sum_{\pm}\pm\int_{\delta}^{\infty}{\frac{1}{x} e^{-\frac{\pi\left(x\pm|\mu|\right)^2}{4Mk\alpha_jz}}}dx
\end{equation*}
for some $y_{1,x},y_{2,x}$ between $0$ and $x$ (in particular, $y_{\ell,x}\in (-\delta,\delta)$).
\end{lemma}
\begin{proof}
We make the change of variables $x\mapsto x+\mu$ in \eqref{defineint} to rewrite the integral as
\begin{equation*}%\label{eqn:intfirst}
\int_{-\infty}^{\infty}\frac{e^{-\frac{\pi x^2}{4Mk\alpha_j z}}}{x-(1+i\varepsilon)\mu} dx
=
\int_{-\infty}^{\infty}\frac{e^{-\frac{\pi \left(x+\mu\right)^2}{4Mk\alpha_j z}}}{x-i\varepsilon\mu} dx.
\end{equation*}
We then split the integral into three pieces as
\[
\int_{-\infty}^{\infty} =\int_{-\delta}^{\delta} + \int_{\delta}^{\infty}+ \int_{-\infty}^{-\delta}=:\mathcal{I}_1+\mathcal{I}_2+\mathcal{I}_3.
\]
To evaluate $\mathcal{I}_1$, note that by Taylor's Theorem, there exist $y_{1,x}$ and $y_{2,x}$ between $0$ and $x$ such that
\[
R_g(x)=R_g(0)+R_g'\!\left(y_{1,x}\right)x \qquad \text{ and }\qquad I_g(x)=I_g(0)+I_g'\!\left(y_{2,x}\right)x.
\]
Therefore
\[
g(x)= e^{-\frac{\pi\mu^2}{4Mk\alpha_jz}}+\left(R_g'\!\left(y_{1,x}\right) + i I_g'\!\left(y_{2,x}\right)\right)x.
\]
Thus
\begin{align}\nonumber
\lim_{\varepsilon\to 0^+}\mathcal{I}_1&=\lim_{\varepsilon\to 0^+}\int_{-\delta}^{\delta} \frac{e^{-\frac{\pi\mu^2}{4Mk\alpha_jz}}+\left(R_g'\!\left(y_{1,x}\right) + i I_g'\!\left(y_{2,x}\right)\right)x}{x-i\varepsilon\mu}dx\\
\label{eqn:I1split}& = \lim_{\varepsilon\to 0^+}e^{-\frac{\pi 	\mu^2 }{4Mk\alpha_jz}} \int_{-\delta}^{\delta} \frac{1}{x-i\varepsilon\mu }dx+ \int_{-\delta}^{\delta}\left( R_g'\!\left(y_{1,x}\right) + i I_g'\!\left(y_{2,x}\right) \right) dx.
\end{align}
The second term on the right-hand side of \eqref{eqn:I1split} is precisely the second term in the claim.

Evaluating the integral explicitly, the first term in \eqref{eqn:I1split} equals
\begin{equation}\label{eqn:Logs}
e^{-\frac{\pi \mu^2}{4Mk\alpha_jz}} \lim_{\varepsilon\to 0^+}\int_{-\delta}^{\delta}\frac{1}{x-i\varepsilon\mu}dx  =e^{-\frac{\pi\mu^2}{4Mk\alpha_jz}}\lim_{\varepsilon\to 0^+}\left(\Log\left(\delta-i\varepsilon\mu\right)-\Log\left(-\delta-i\varepsilon\mu\right)\right).
\end{equation}
Here and throughout, $\Log$ denotes the principal branch of the complex logarithm.
We then evaluate, using the fact that $\mu\neq 0$,
\begin{align*}
\lim_{\varepsilon\to 0^+}
\Log\left(\delta-i\varepsilon
\mu
\right)&=\log(\delta),\\
\lim_{\varepsilon\to 0^+}
\Log\left(-\delta-i\varepsilon\mu\right)&=
\begin{cases}
\Log(-\delta)=\log(\delta) +\pi i&\text{if }
\mu
<0,\\
\Log(-\delta)-2\pi i=\log(\delta)-\pi i&\text{if }\mu>0.
\end{cases}
\end{align*}
Therefore \eqref{eqn:Logs} becomes
\[
\pi i \sgn(\mu)e^{-\frac{\pi\mu^2}{4Mk\alpha_jz}}.
\]
Since the paths of integration in  $\mathcal{I}_2$ and $\mathcal{I}_3$ do not go through zero, we can plug in $\varepsilon=0$ to obtain
\begin{equation}\label{eqn:I2+I3}
\lim_{\varepsilon\to 0^+}\left(\mathcal{I}_2+\mathcal{I}_3\right) = \int_{\delta}^{\infty} \frac{1}{x}{e^{-\frac{\pi (x+\mu)^2}{4Mk\alpha_j z}}} dx +  \int_{-\infty}^{-\delta} \frac{1}{x}{e^{-\frac{ \pi(x+
\mu
)^2 }{4Mk\alpha_j z}}} dx.
\end{equation}
Making the change of variables $x\mapsto -x$ in the second integral, we see that \eqref{eqn:I2+I3} becomes
\[
 \int_{\delta}^{\infty} \frac{1}{x}e^{-\frac{\pi(x+\mu)^2}{4Mk\alpha_j z}}dx
-\int_{\delta}^{\infty} \frac{1}{x} e^{-\frac{ \pi(x-\mu)^2}{4Mk\alpha_j z}} dx
 =\sgn(\mu)\sum_{\pm}\pm\int_{\delta}^{\infty}\frac{1}{x} e^{{-\frac{\pi\left(x\pm|
\mu
|\right)^2}{4Mk\alpha_jz}}}dx.\qedhere
\]
\end{proof}

\subsection{Asymptotics for $\mathcal{I}(\mu,k;z)$}
The main result in this subsection is the following approximation of $\mathcal{I}(\mu,k;z)$.
\begin{proposition}\label{prop:intbound}
If $1\leq k\leq N$ and $|\Phi|\leq \frac{1}{kN}$, then
for $0<\delta<\frac{|
\mu
|}{2}$ we have, for some $c>0$
\[
\mathcal{I}(\mu,k;z)=- \frac{2\sqrt{Mk\alpha_jz}}{\mu}+O\left(\frac{k^{\frac 32}|z|^{\frac 32}}{|\mu|^3}+\left(1 + \frac{|\mu|\delta}{k|z|}+\log\left(\frac{|\mu|}{\delta}\right)\right)e^{-\frac{c\mu^2}{k}\re{\frac{1}{z}}}\right).
\]
\end{proposition}
Before proving Proposition \ref{prop:intbound}, we  approximate the third term from Lemma \ref{lem:inteval}. We set
$
A:=\frac{\pi\mu^2}{4Mk\alpha_j|z|}
$
and make the change of variables $x\mapsto |\mu|x$ to obtain for the third term in Lemma \ref{lem:inteval}
\begin{equation}\label{eqn:thirdterm2}
\sgn(\mu)\sum_{\pm}\pm \int_{\frac{\delta}{|\mu|}}^{\infty}\frac{1}{x} e^{-\frac{A|z|}{z}(x\pm 1)^2}dx.
\end{equation}
We split the integral at $x=\frac 12$. To approximate the contribution from $x\geq \frac{1}{2}$, we define for $d\in\N_0$
\begin{equation}\label{eqn:Jddef}
\mathcal{J}_{d,\pm}:=C_d\left(\frac{z}{2A|z|}\right)^{d-1}\int_{\frac{1}{2}}^{\infty} \frac{1}{x^d}e^{-A\frac{|z|}{z} (x\pm 1)^2} dx,
\end{equation}
where
\begin{equation*}
C_d:= \begin{cases} (d-1)!&\text{if }d\geq1,\\ 1&\text{if }d=0.\end{cases}
\end{equation*}
Note that $\mathcal{J}_{1,\pm}$ is the contribution from $x\geq \frac{1}{2}$ to the integral in \eqref{eqn:thirdterm2}. The following trivial bound for
 $\mathcal{J}_{d,\pm}$ follows immediately
by bringing the absolute value inside the integral.
\begin{lemma}\label{lem:Jbndtrivial}
For $d\in\N_0$, we have
\[
\left|\mathcal{J}_{d,\pm}\right|\leq \frac{2\sqrt{\pi}C_d A^{\frac{1}{2}-d}}{\sqrt{|z|\re{\frac{1}{z}}}}.
\]
\end{lemma}
To obtain a better approximation for $\mathcal{J}_{d,\pm}$, we next relate
 $\mathcal{J}_{d,\pm}$
with
 $\mathcal{J}_{d+1,\pm}$ and $\mathcal{J}_{d-1,\pm}$.
\begin{lemma}\label{lem:Jdtrick}
For $d\in\N$, we have
\[
\mathcal{J}_{d,\pm} =\mp\left(
-(d-1)!
\left(\frac{z}{A|z|}\right)^{d} e^{-\frac{A|z|}{z}\left(\frac{1}{2}\pm 1\right)^2} +\mathcal{J}_{d+1,\pm} +
\max\{d-1,1\}
  \frac{z}{2A|z|}\mathcal{J}_{d-1,\pm}\right).
\]
\end{lemma}
\begin{proof}
We first rewrite
\begin{equation}\label{eqn:Jdtrick}
\mathcal{J}_{d,\pm} = \mathcal{J}_{d,\pm} \pm \frac{C_{d}}{C_{d-1}} \frac{z}{2A|z|}\mathcal{J}_{d-1,\pm}\mp  \frac{C_{d}}{C_{d-1}}  \frac{z}{2A|z|}\mathcal{J}_{d-1,\pm}.
\end{equation}
Using integration by parts, the first two terms in \eqref{eqn:Jdtrick} equal
\begin{equation*}
\pm C_d\left(\frac{z}{2A|z|}\right)^{d-1}\int_{\frac{1}{2}}^{\infty}\frac{1}{x^d}{(x\pm 1)}e^{-A\frac{|z|}{z} (x\pm 1)^2}dx=\pm C_d\left(\frac{z}{A|z|}\right)^{d} e^{-\frac{A|z|}{z}\left(\frac{1}{2}\pm 1\right)^2} \mp \mathcal{J}_{d+1,\pm}.
\end{equation*}
Plugging back into \eqref{eqn:Jdtrick} and using $C_d=(d-1)!$ and $\frac{C_{d}}{C_{d-1}}=\max\{d-1,1\}$ yields the claim.
\end{proof}

We also require an approximation for $\mathcal{J}_{0,\pm}$.
\begin{lemma}\label{lem:J0eval}
There exists $c>0$ such that
\[
\mathcal{J}_{0,\pm} =2\delta_{\pm  1=-1} \sqrt{\frac{\pi A|z|}{z}}+O\left(\frac{\sqrt{A} e^{-c A|z|\re{\frac{1}{z}}}}{\sqrt{|z|\re{\frac{1}{z}}}}\right).
\]
\end{lemma}
\begin{proof}
We first make the change of variables $x\mapsto x\mp 1$ in \eqref{eqn:Jddef} to obtain that
\[
\mathcal{J}_{0,\pm}=\frac{2A|z|}{z}\int_{\frac{1}{2}\pm 1}^{\infty}e^{-\frac{A|z|}{z}x^2}dx.
\]
For $\pm 1=-1$, we rewrite this as
\begin{align*}
\mathcal{J}_{0,-}&=\frac{2A|z|}{z}\left(\int_{-\infty}^{\infty}e^{-\frac{A|z|}{z}x^2}dx-\int_{\frac{1}{2}}^{\infty}e^{-\frac{A|z|}{z}x^2}dx\right)
\\&=\frac{2A|z|}{z}\int_{-\infty}^{\infty}e^{-\frac{A|z|}{z}x^2}dx+O\left(A\int_{\frac{1}{2}}^{\infty}e^{-A|z|\re{\frac{1}{z}} x^2}dx\right).
\end{align*}
Hence we have
\[
\mathcal{J}_{0,\pm}=2\delta_{\pm 1=-1} \frac{A|z|}{z}\int_{-\infty}^{\infty}e^{-\frac{A|z|}{z}x^2}dx+O\left(A\int_{1\pm \frac{1}{2}}^{\infty}e^{-A|z|\re{\frac{1}{z}} x^2}dx\right).
\]
Noting that $\operatorname{Re}(\frac{1}{z})>0$, we then bound
\begin{align*}
\int_{1\pm \frac{1}{2}}^{\infty}e^{-A|z|\re{\frac{1}{z}} x^2}dx&\leq
\frac{\sqrt{\pi}e^{-\frac{1}{4}A|z|\re{\frac{1}{z}}}}{2\sqrt{A|z|\re{\frac{1}{z}}}}.
\end{align*}
The claim follows, evaluating
\[
\int_{-\infty}^{\infty} e^{-\frac{A|z|}{z} x^2} dx = \sqrt{\frac{\pi z}{A |z|}}. \qedhere
\]
\end{proof}
\noindent

We next combine Lemmas \ref{lem:Jdtrick} and \ref{lem:J0eval} to obtain an approximation for $J_{1,\pm}$. To compare the asymptotic growth of different terms, we note that by \eqref{eqn:PhiBound} and the fact that $k\leq N$, one obtains
\begin{align}\label{realbound}
\sqrt{\frac{\re{\frac{1}{z}}}{k}}&=\frac{1}{kN\sqrt{\frac{1}{N^4}+\Phi^2}}\geq \frac{1}{\sqrt{2}},\\
\label{eqn:k|z|bound}
\frac{k^2}{N^2}&\leq k|z|=k^2\sqrt{\frac{1}{N^4}+\Phi^2}\leq \sqrt{2}.
% \frac{k^2}{kN}=\sqrt{2}\frac{k}{N}\leq \sqrt{2}.
\end{align}
\begin{lemma}\label{lem:J1bnd}
If $1\leq k\leq N$ and $|\Phi|<\frac{1}{kN}$, then we have
\[
 \mathcal{J}_{1,\pm}=
\delta_{\pm 1 =-1}
 \sqrt{\frac{\pi z}{A|z|}}+O\left(A^{-\frac{3}{2}}+e^{-cA|z|\re{\frac{1}{z}}}\right).
\]

\end{lemma}
\begin{proof}
By Lemma \ref{lem:Jdtrick} with $d=1$, we have
\begin{equation*}\label{eqn:J1evalstart}
\mathcal{J}_{1,\pm} =\mp \left(-\frac{z}{A|z|} e^{-\frac{A|z|}{z}\left(\frac{1}{2}\pm 1\right)^2} +\mathcal{J}_{2,\pm} +  \frac{z}{2A|z|}\mathcal{J}_{0,\pm}\right).
\end{equation*}
We then plug in Lemma \ref{lem:Jdtrick}
again twice (once with $d=2$ and then once with $d=1$)
 to obtain that
\begin{align*}
\nonumber
\mathcal{J}_{1,\pm}
&
=\mp\Bigg( \left(-\frac{z}{A|z|} +\left(-\frac{1}{2}\pm 1\right)\left(\frac{z}{A|z|}\right)^2\right) e^{-\frac{A|z|}{z}\left(\frac{1}{2}\pm 1\right)^2} \mp \mathcal{J}_{3,\pm}\\
\label{eqn:J1eval3}
&\hspace{2.5in}+ \frac{z}{2A|z|} \mathcal{J}_{2,\pm} +\left(\frac{z}{2A|z|}+\left(\frac{z}{2A|z|}\right)^2\right)\mathcal{J}_{0,\pm}\Bigg).
\end{align*}
The first term can be bounded against
\[
O\left(\left(\frac{1}{A}+\frac{1}{A^2}\right)e^{-c A|z|\re{\frac{1}{z}}}\right)=O\left(\frac{1}{A}e^{-c A|z|\re{\frac{1}{z}}}\right),
\]
using that $A\gg 1$ by \eqref{eqn:k|z|bound}. Moreover, by Lemma \ref{lem:Jbndtrivial}, we have
\[
\left|\mathcal{J}_{3,\pm}\right|,\
\frac{z}{2A|z|} \left|\mathcal{J}_{2,\pm}\right|
\ll \frac{A^{-2}}{\sqrt{A|z|\re{\frac{1}{z}}}}.
\]
For the terms with $\mathcal{J}_{0,\pm}$, we use Lemma \ref{lem:J0eval} to approximate these by
\begin{align*}
\mp \delta_{\pm  1=-1}\sqrt{\frac{\pi z}{A|z|}}+O\left(A^{-\frac 32}+
\frac{e^{-cA|z|\re{\frac 1z}}}{\sqrt{A|z|\re{\frac{1}{z}}}}\right).
\end{align*}

Noting that $\mp \delta_{\pm 1=-1}=\delta_{\pm 1=-1}$, this gives
\[
\mathcal{J}_{1,\pm}=\delta_{\pm 1=-1}\sqrt{\frac{\pi z}{A|z|}} +O\left(A^{-\frac{3}{2}} + \left(\frac{1}{A}+\frac{1}{\sqrt{A|z|\re{\frac{1}{z}}}}\right)e^{-c A|z|\re{\frac{1}{z}}} + \frac{A^{-2}}{\sqrt{A|z|\re{\frac{1}{z}}}}\right).
\]
We then use \eqref{realbound} and the trivial bound $|z|\operatorname{Re}(\frac{1}{z})\leq 1$ to compare the $O$-terms, obtaining
\[
\frac{A^{-2}}{\sqrt{A|z|\re{\frac{1}{z}}}}\ll A^{-\frac{3}{2}}\qquad\text{ and }\qquad \frac{1}{A}\ll \frac{1}{\sqrt{A|z|\re{\frac{1}{z}}}}\ll 1.
\]
This gives the claim.
\end{proof}

We are now ready to prove Proposition \ref{prop:intbound}.
\begin{proof}[Proof of Proposition \ref{prop:intbound}]
The first term in Lemma \ref{lem:inteval} yields the second error term in Proposition \ref{prop:intbound}. For the second term in Lemma \ref{lem:inteval}, we note that
\begin{align*}
R_g(y)&=e^{-\frac{\pi (y+\mu)^2}{4Mk\alpha_j}\re{\frac{1}{z}}}\cos\left(-\frac{\pi(y+\mu)^2}{4Mk\alpha_j}\mathrm{Im}\hspace{-.1cm}\left(\frac 1z\right)\right),\\
I_g(y)&=e^{-\frac{\pi (y+\mu)^2}{4Mk\alpha_j}\re{\frac{1}{z}}}\sin\left(-\frac{\pi(y+\mu)^2}{4Mk\alpha_j}\mathrm{Im}\hspace{-.1cm}\left(\frac 1z\right)\right)
\end{align*}
and then explicitly take the derivatives and bound $|\mathrm{Re}(z)|$, $|\mathrm{Im}(z)|<|z|$, and the absolute value of the sines and cosines that occur against $1$. This yields
\[
\left|R_g'\left(y_{1,x}\right)+iI_g'\left(y_{2,x}\right)\right|\leq \frac{\pi}{Mk\alpha_j|z|}\sum_{\ell=1}^2 \left|y_{\ell,x}+\mu\right|e^{-\frac{\pi (y_{\ell,x}+\mu)^2}{4Mk\alpha_j}\re{\frac{1}{z}}}.
\]
To bound the right-hand side, we use  $y_{\ell,x}<\delta<\frac{|\mu|}{2}$ to conclude that $\frac{|\mu|}{2}\leq  |y_{\ell,x}+\mu|\leq \frac{3|\mu|}{2}$. Since $\operatorname{Re}(\frac{1}{z})>0$, the second term in Lemma \ref{lem:inteval} contributes the third error term in Proposition \ref{prop:intbound}. We rewrite the third term in Lemma \ref{lem:inteval} as in \eqref{eqn:thirdterm2} and split the integral in \eqref{eqn:thirdterm2} at $\frac{1}{2}$. For $\frac{\delta}{|\mu|}\leq x\leq \frac{1}{2}$, we bring the absolute value inside and note that for $x\leq \frac{1}{2}$ we have $|x\pm 1|\geq \frac{1}{2}$ to bound
\begin{equation}\label{eqn:smallxintparts}
  \int_{\frac{\delta}{|\mu|}}^{\frac{1}{2}} \frac{1}{x}e^{-\frac{A|z|}{z} (x\pm 1)^2}dx\leq e^{-\frac{A|z|}{4} \re{\frac{1}{z}}}\int_{\frac{\delta}{|\mu|}}^{\frac{1}{2}} \frac{1}{x} dx \ll \left( 1+\log\left(\frac{|\mu|}{\delta}\right)\right)e^{-\frac{A|z|}{4} \re{\frac{1}{z}}}.
\end{equation}

We next turn to the contribution from $x\geq \frac 12$. By Lemma \ref{lem:J1bnd}, we have \begin{equation}\label{eqn:pluginJ0}
\sgn(\mu)\sum_{\pm}\pm\mathcal{J}_{1,\pm}=-\sgn(\mu)\sqrt{\frac{\pi z}{A|z|}} +O\left(A^{-\frac{3}{2}} + e^{-c A|z|\re{\frac{1}{z}}}\right).
\end{equation}
As noted below \eqref{eqn:Jddef}, $\mathcal{J}_{1,\pm}$ is precisely the contribution from $x\geq \frac{1}{2}$ to the integral in \eqref{eqn:thirdterm2}. Therefore, combining \eqref{eqn:pluginJ0} with \eqref{eqn:smallxintparts} yields
\[
\sgn(\mu)\sum_{\pm}\pm\int_{\frac{\delta}{|\mu|}}^{\infty} \frac{1}{x}e^{-\frac{A|z|}{z}(x\pm 1)^2}dx = -\sgn(\mu)\sqrt{\frac{\pi z}{A|z|}} +O\left(A^{-\frac{3}{2}} + \left(1+\log\left(\frac{|\mu|}{\delta}\right)\right) e^{-c A|z|\re{\frac{1}{z}}}\right).
\]
 Plugging in $A=\frac{\pi \mu^2}{4Mk\alpha_j|z|}$ gives that this equals
\begin{align*}
 -\frac{2\sqrt{Mk\alpha_j z}}{\mu} +O\left(\frac{k^{\frac 32}|z|^{\frac 32}}{|\mu|^3}+\left(1+\log\left(\frac{|\mu|}{\delta}\right)\right)e^{-\frac{c\mu^2}{k}\re{\frac{1}{z}}}\right),
\end{align*}
where the value of $c$ is changed from the previous line.
These correspond to the main term and the first, second, and fourth error terms in Proposition \ref{prop:intbound}.
\end{proof}
We directly obtain the following corollary by choosing $\delta:=\frac{k|z|}{2\sqrt{2}|\mu|}$ in Proposition \ref{prop:intbound}.
\begin{corollary}\label{cor:intbound}
We have, for some $c>0$
\[
\mathcal{I}(\mu,k;z)=- \frac{2\sqrt{Mk\alpha_jz}}{\mu}+O\left(\frac{k^{\frac 32}|z|^{\frac 32}}{|\mu|^{3}}+\log\left(\frac{\mu^2}{k|z|}\right)e^{-\frac{c\mu^2}{k}\re{\frac{1}{z}}}\right).
\]
\end{corollary}

\subsection{Summing $\mathcal{I}(\mu,k;z)$}

We next approximate the sum over $\nu$ in the second term of Lemma \ref{lem:Fmodular}.
\begin{lemma}\label{lem:intsumbound}
There exists $c>0$ such that for all $0<k\leq N$ and $\ell\in\mathcal{L}_{Mk}$ we have
\begin{align}
\label{eqn:intsumboundmain}
\sideset{}{^*}\sum_{\nu\geq 0}\sum_{\pm} \mathcal{I}\left(\ell\pm 2Mk\nu,k;z\right)&=-\pi \sqrt{\tfrac{\alpha_jz}{Mk}}\cot\left(\tfrac{\pi\ell}{2Mk}\right)\!+\! O\left(\tfrac{k^{\frac{3}{2}}|z|^{\frac{3}{2}}}{|\ell|^3}\right)\!+\!O\left(\frac{1+\left|\log(k|z|)\right|}{e^{\frac{c\ell^2}{k}\re{\frac{1}{z}}}}\right)\\
\label{eqn:intsumboundO2}
&=O\left(\tfrac{\sqrt{k|z|}}{|\ell|}+\left(1+\left|\log(k|z|)\right|\right)e^{-\frac{c\ell^2}{k}\re{\frac{1}{z}}}\right)\\
\label{eqn:intsumboundO3}
&=O\left(\tfrac{n^{\varepsilon}}{|\ell|}\right).
\end{align}
\end{lemma}
\begin{rem}
 Note that the first term on the right-hand side
of \eqref{eqn:intsumboundmain}
 is always finite because $1-Mk\leq\ell \leq Mk$ with $\ell\neq 0$ implies that the parameter is never an integer multiple of $\pi$.
\end{rem}
\begin{proof}[Proof of Lemma \ref{lem:intsumbound}]
Plugging Corollary \ref{cor:intbound} with $\mu=\ell\pm 2Mk\nu$ into the left-hand side of Lemma \ref{lem:intsumbound} and using
\begin{equation*} \label{415}
\pi \cot(\pi x) = \lim\limits_{N \rightarrow \infty} \left(\frac{1}{x} + \sum_{n=1}^{N} \left(\frac{1}{x+n}+\frac{1}{x-n}\right)\right),
\end{equation*}
the main term in \eqref{eqn:intsumboundmain} becomes the claimed main term.

To obtain \eqref{eqn:intsumboundmain}, we are left to bound the error terms. Note that since
 $1-Mk\leq \ell \leq Mk$ (with $\ell\neq 0$), we have $\frac{2Mk}{|\ell|}\geq 2$. We conclude that since $|\frac{2Mk}{\ell}\nu-1|\geq 2\nu-1\geq \nu$ for $\nu\geq 1$, the sum of the first $O$-term in Corollary \ref{cor:intbound} is
\[
\sideset{}{^*}{\sum}_{\nu\geq 0}\sum_{\pm} \frac{1}{|\ell\pm 2Mk\nu|^3}\leq \frac{1}{|\ell|^3}+ \frac{2}{|\ell|^3} \sum_{ \nu\geq 1} \frac{1}{\nu^3}\ll \frac{1}{|\ell|^3},
\]
yielding the first error-term in the lemma.
For the final error-term, we write
\begin{multline}\label{eqn:absconv}
\sideset{}{^*}\sum_{\nu\geq 0}\sum_{\pm}\log\left(\frac{|\ell\pm 2Mk\nu|^2}{k|z|}\right)e^{-\frac{c(\ell\pm 2Mk\nu)^2}{k}\re{\frac{1}{z}}}\\
=-\log\left(k|z|\right)\sideset{}{^*}\sum_{\nu\geq 0}\sum_{\pm} e^{-\frac{c(\ell\pm 2Mk\nu)^2}{k}\re{\frac{1}{z}}}  +2 \sideset{}{^*}\sum_{\nu\geq 0}\sum_{\pm}\log\left(|\ell\pm 2Mk\nu|\right) e^{-\frac{c(\ell\pm 2Mk\nu)^2}{k}\re{\frac{1}{z}}}.
\end{multline}
 Since $1-Mk\leq \ell\leq Mk$, we have $d:=|\ell\pm 2Mk\nu|\geq |\ell|$ for every $\nu$ and the terms in all sums in \eqref{eqn:absconv} are non-negative. Hence we may bound \eqref{eqn:absconv} against a constant multiple of
 \begin{align*}
 e^{-\frac{c\ell^2}{2k}\re{\frac{1}{z}}}\log(k|z|)\sum_{d\geq |\ell|}e^{-\frac{cd^2}{4}} + e^{-\frac{c\ell^2}{2k}\re{\frac{1}{z}}}\sum_{d\geq |\ell|}\log(d) e^{-\frac{cd^2}{4}}.
\end{align*}
Each of the sums is absolutely convergent and may be bounded by the corresponding sum with $|\ell|=1$, giving a uniform estimate independent of $\ell$. We obtain \eqref{eqn:intsumboundmain} by bounding it against
\[
\left(1+\left|\log\left(k|z|\right)\right|\right)e^{-\frac{c\ell^2}{2k}\re{\frac{1}{z}}}.
\]

The approximation \eqref{eqn:intsumboundO2} follows by showing that
\begin{align*}
\label{eqn:errortomain1}
\frac{k^{\frac{3}{2}}|z|^{\frac{3}{2}}}{|\ell|^3}, \qquad
\sqrt{\frac{|z|}{k}}\left|\cot\left(\frac{\pi \ell}{2Mk}\right)\right|\ll \frac{\sqrt{k|z|}}{\ell}.
\end{align*}
Finally \eqref{eqn:intsumboundO3} follows by \eqref{realbound}, \eqref{eqn:k|z|bound}, and \eqref{eqn:intsumboundO2}.
\end{proof}

\section{Proof of Theorem \ref{thm:rNrZ}}\label{sec:CircleMethod}

\subsection{Kloosterman's fundamental lemma}
We require a slightly modified version of Kloosterman's fundamental lemma \cite[Lemma 6]{Kloosterman}. For this, we need certain relations for Gauss sums that follow from their multiplicativity and well-known evaluations of them modulo prime powers; see \cite{BEW} for details. For $\bm{\alpha}\in \N^4$, $k,M,R\in\N$, $\bm{\nu}\in\Z^4$ and $j\in\{1,2,3,4\}$, we set $a_j:=2M\alpha_j$, $b_j:=2\alpha_j r$, and $c_j:=\nu_j$ and evaluate $G(a_jh,b_jh+c_j;k)$ for $h\in\Z$. To state the result, set $d_j:=\gcd(a_j,k)$, and $\eta_j:=\frac{a_j}{d_j}$ and write $d_j=2^{r_j}\delta_j$ and $k_j:=\frac{k}{d_j}=2^{m_j}\ell_j$ with $\delta_j$ and $\ell_j$ odd.
Defining
\begin{multline*}
\!\b_j:=\begin{cases}
4\eta_j & \text{if } m_j=0,\\ 8\eta_j & \text{if } m_j=1,\\ \eta_j & \text{if } m_j\geq2,
\end{cases}
\quad \!
\g_j:=\begin{cases}
d_j & \text{if } m_j=0 \text{ and } d_j\not\equiv2\pmod{4},\\
4d_j & \text{if } m_j=0 \text{ and } d_j\equiv2\pmod{4},\\
2d_j & \text{if } m_j=1 \text{ and } d_j\equiv0\pmod{2},\\
8d_j & \text{if } m_j=1 \text{ and } d_j\equiv1\pmod{2},\\
 4d_j & \text{if } m_j\geq2,
\end{cases}\quad \!
\w_j:=\begin{cases}
\ell_j & \text{if } m_j=0,1,\\
\frac{4k}{d_j} & \text{if } m_j\geq2,
\end{cases}
\\
f_j(h,k):=
\begin{cases}
\e_{\ell_j}\sqrt{d_j}\lp\dfrac{\eta_j}{\ell_j}\rp\lp\dfrac{h}{\delta_j}\rp \lp\dfrac{2^{r}}{h}\rp e^{-\frac{4\pi ib_jc_j}{\ell_j d_j^2}[4\eta_j]_{\ell_j}} & \text{if } m_j=0 \text{ and } d_j|(b_jh+c_j),\\
 \e_{\ell_j}\sqrt{2d_j}\lp\dfrac{2\eta_j}{\ell_j}\rp \lp\dfrac{h}{\d_j}\rp \lp\dfrac{2^{r}}{h}\rp e^{-\frac{4\pi ib_jc_j}{\ell_j d_j^2}[8\eta_j]_{\ell_j}} & \text{if } m_j=1,\ d_j|(b_jh+c_j),\\& \text{and } 2d_j\nmid (b_jh+c_j),\\
 \frac{\e_{\ell_j}(1+i)}{\e_{\ell_j \eta_jh}}\sqrt{d_j}\lp\dfrac{\eta_j}{\ell_j}\rp\lp\dfrac{2^{m_j}}{\eta_j}\rp \lp\dfrac{h}{\d_j}\rp\lp\dfrac{2^{r_j}}{h}\rp e^{-\frac{\pi ib_jc_j}{k_jd_j^2}[\eta_j]_{4k_j}} & \text{if } m_j\geq2 \text{ and } 2d_j|(b_jh+c_j),\\
 0 & \text{otherwise},
\end{cases}
\end{multline*}
and letting $0\le[a]_b<b$ denote the inverse of $a\Pmod b$ if $\gcd(a,b)=1$ and $b>0$, a long but straightforward calculation yields the following.

\begin{lemma}\label{lem:GaussSum}
For $j\in\{1,2,3,4\}$, $\nu_j,R\in\Z$, $M\in\N$, $h\in\Z$, $k=2^r\kzer\in\N$ with $\kzer\in\N$ odd, $\gcd(h,k)=1$, and $\bm{\alpha}\in\N^4$, the following hold:
\begin{enumerate}[leftmargin=*,label=\rm(\arabic*)]
\item We have 
	\begin{equation*}
		G(2M\alpha_jh, 2\alpha_jh R+\nu_j;k)=f_j(h,k)\sqrt{k}\lp\dfrac{h}{\kzer}\rp\lp\dfrac{2^{r}}{h}\rp e^{-\frac{2\pi i}{\w_j d_j^2}[\b_j]_{\w_j}\lp4\alpha_j^2R^2h+\nu_j^2[h]_{\w_jd_j^2}\rp}.
	\end{equation*}
\item For $h_1\equiv h_2\pmod{\g_j}$, we have $f_j(h_1,k)=f_j(h_2,k)$.
\item Independent of $k$ and $n$, we have $d_j,\b_j,\g_j,\dfrac{k}{\w_jd_j},f_j(h,k)=O(1)$.
\item We have $\w_jd_j|4k$. Moreover, $\g_j=d_j$ if $k$ is odd and $\g_j|8d_j$ if $k$ is even.
\end{enumerate}
\end{lemma}

To state the modified version of Kloosterman's fundamental lemma, we note that for each $0\leq h<k$ with $\gcd(h,k)=1$, there exists a unique $\varrho(h)=\varrho_{k}(h)$ with $0<\varrho(h)\leq k$ for which
\begin{equation*}\label{cong}
h\left(N+\varrho(h)\right)\equiv -1\pmod{k}.
\end{equation*}
\begin{lemma}\label{lem:ModifiedKloostermanFund}
For $\bm{\nu}\in\Z^4$, $k\in\N$, $0<\varrho< k$,  and $n,R\in\Z$, we have
\[
\left|\sum_{\substack{
0\leq h<k\\ \gcd(h,k)=1\\
\varrho(h)\leq \varrho}} e^{-\frac{2\pi inh}{k}} \prod_{j=1}^{4} G(2M\alpha_jh,2\alpha_j hR+\nu_j;k)\right|=O\left(\gcd(P(n),k)^{\frac{1}{4}}k^{2+\frac{7}{8}+\varepsilon}\right),
\]
where $P(n)$ is a polynomial in $n$ independent of $\bm\nu$, $k$, and $\varrho$. Here the $O$-constant is absolute (and in particular independent of $\fbm\bm{\nu}$ and $\varrho$).
\end{lemma}
\begin{proof}

We adopt the definitions of $\b_j,\g_j,\w_j,f_j$ from Lemma \ref{lem:GaussSum}. Let  $\GG:=\lcm\left(\g_1,\g_2,\g_3,\g_4\right)$ and
 \[
\OO:=\left\{\begin{array}{ll}k\lcm\left(d_1^2,d_2^2,d_3^2,d_4^2\right) & \text{if } k \text{ is odd},\vspace{.1cm}\\ 4k\lcm\left(d_1^2,d_2^2,d_3^2,d_4^2\right) & \text{if } k \text{ is even}.\end{array}\right.
\]
Then, by Lemma \ref{lem:GaussSum} (3), we have $\GG=O(1)$ and $\OO=O(k)$. By Lemma \ref{lem:GaussSum} (4), we have $\GG\,|\,\OO$. Since $k\,|\,\OO\,|\,k^3$ if $k$ is odd (since $d_j\mid k$) and $k\,|\,\OO\,|\,4k^3$ if $k$ is even, $k$ and $\OO$ have exactly the same prime factors. Thus $\gcd(h,k)=1$ if and only if $\gcd(h,\OO)=1$, and for each $0\leq h<k$ with $\gcd(h,k)=1$, there is exactly one $0\leq h_2<\OO$ with $\varrho_{\OO}(h_2)=\varrho_{k}(h)$ and this $h_2$ satisfies $h_2\equiv h\pmod{k}$. Hence
\[
\left|\sum_{\substack{0\leq h<k\\ \gcd(h,k)=1\\ \varrho_k(h)\leq \varrho}} e^{-\frac{2\pi inh}{k}} \prod_{j=1}^{4} G(2M\alpha_jh,2\alpha_j hR+\nu_j;k)\right|=\left|\sum_{\substack{0\leq h<\OO\\ \gcd(h,\OO)=1\\ \varrho_\OO(h)\leq \varrho}} e^{-\frac{2\pi inh}{k}} \prod_{j=1}^{4} G(2M\alpha_jhR,2\alpha_j h+\nu_j;k)\right|.
\]
By Lemma \ref{lem:GaussSum} (1), (2), we may write the right-hand side as
\begin{align}
\nonumber&\left|\sum_{\substack{0\leq h<\OO\\ \gcd(h,\OO)=1\\ \varrho_\OO(h)\leq \varrho}} e^{-\frac{2\pi inh}{k}} \prod_{j=1}^{4} f_j(h,k)\sqrt{k}\lp\dfrac{h}{k}\rp e^{-\frac{2\pi i}{\w_j d_j^2}[\b_j]_{\w_j}\lp4\alpha_j^2R^2h+\nu_j^2[h]_{w_jd_j^2}\rp}\right|\\
&\quad\label{eqn:FundEqual}=k^2\left|\sum_{r=1}^{\GG} \prod_{j=1}^4 f_j(r,k) \sum_{\substack{0\leq h<\OO\\ \gcd(h,\OO)=1\\ \varrho_\OO(h)\leq \varrho\\ h\equiv r\pmod{\GG}}} e^{-\frac{2\pi i}{\OO}\lp\lp\frac{n\OO}{k}+\sum_{j=1}^4 \frac{4\alpha_j^2R^2\OO}{\w_jd_j^2}[\b_j]_{\w_j}\rp h+\lp\sum_{j=1}^4 \frac{\nu_j^2\OO}{\w_jd_j^2}[\b_j]_{\w_j}\rp[h]_{\OO}\rp}\right|.
\end{align}
By \cite[Lemma 5]{Kloosterman} and Lemma \ref{lem:GaussSum} (3), we can bound \eqref{eqn:FundEqual} against
\begin{equation}\label{eqn:FundBound}
\ll k^2\OO^{\frac{7}{8}+\e}\gcd\lp\frac{n\OO}{k}+\sum_{j=1}^4 \frac{4\alpha_j^2R^2\OO}{\w_jd_j^2}[\b_j]_{\w_j}, \OO\rp^\frac14.
\end{equation}
Since $k\mid \OO$ and $\OO=O(k)$, we have $\frac{\OO}{k}=O(1)$, and the gcd in \eqref{eqn:FundBound} can be bounded against
\begin{align}
\nonumber& \gcd\lp\frac{n\OO}{k}+\sum_{j=1}^4 \frac{4\alpha_j^2R^2\OO}{\w_jd_j^2}[\b_j]_{\w_j}, k\rp \gcd\lp\frac{n\OO}{k}+\sum_{j=1}^4 \frac{4\alpha_j^2R^2\OO}{\w_jd_j^2}[\b_j]_{\w_j}, \dfrac{\OO}{k}\rp\\
\nonumber&\leq \gcd\lp\prod_{j=1}^4\b_j\lp\frac{n\OO}{k}+\sum_{j=1}^4 \frac{4\alpha_j^2R^2\OO}{\w_jd_j^2}[\b_j]_{\w_j}\rp,k\rp  \frac{\OO}{k}\\
&\ll \gcd\lp\dfrac{n\OO}{k}\prod_{j=1}^4\b_j +\sum_{j=1}^4 \frac{4\alpha_j^2R^2\OO}{\w_jd_j^2}\prod_{\ell\in\{1,2,3,4\}\backslash\{j\}}\b_\ell,k\rp\label{eqn:FundGCDBound}=\gcd(A_kn+B_k,k),
\end{align}
where
\begin{align*}
	A_k&:=\dfrac{\OO}{k}\prod_{j=1}^4\b_j=c_k\lcm\left(d_1^2,d_2^2,d_3^2,d_4^2\right)\prod_{j=1}^4\b_j,\\
	B_k&:=\sum_{j=1}^4 \frac{4\alpha_j^2R^2\OO}{\w_jd_j^2}\prod_{\ell\in\{1,2,3,4\}\backslash\{j\}}\b_{\ell}=c_k\sum_{j=1}^4 \frac{4\alpha_j^2R^2\lcm\left(d_1^2,d_2^2,d_3^2,d_4^2\right)}{d_j}\dfrac{k}{\w_jd_j}\prod_{\ell\in\{1,2,3,4\}\backslash\{j\}}\b_{\ell},
\end{align*}
with
\[
c_k:=\left\{\begin{array}{ll}1 & \text{if } k \text{ is odd},\\ 4 & \text{if } k \text{ is even}.\end{array}\right.
\]
The values of $A_k$ and $B_k$ only depend on the values of $R$, $\alpha_j$, $c_k$, $d_j$, $\b_j$, and $\frac{k}{\w_jd_j}$. By Lemma \ref{lem:GaussSum} (3), $d_j$, $\b_j$, and $\frac{k}{\w_jd_j}$ only have a finite number of possible values, independent of $k$ and $n$. As $c_k$ also takes only two values, there exists a finite set $\cS$, independent of $k$ and $n$, containing all possible pairs of $(A_k,B_k)$. Applying \eqref{eqn:FundGCDBound} to \eqref{eqn:FundBound}, we can bound \eqref{eqn:FundEqual} against (up to a uniform constant)
\[
\gcd\lp A_kn+B_k,k\rp^\frac14 k^{2+\frac{7}{8}+\e}
\ll \gcd\lp P(n),k\rp^\frac14 k^{2+\frac{7}{8}+\e},
\]
where $P(n):=\prod_{(A,B)\in\cS}(An+B)$ is a polynomial on $n$ independent of $\bm\nu$, $k$, and $\varrho$.
\end{proof}
One obtains the value of $\varrho(h)$ by \eqref{eqn:adjacent} and \eqref{eqn:varrhojbnd}.
\begin{lemma}\label{lem:rhoh}
 We have
\[
\varrho(h) =\varrho_{k,1}(h).
\]
\end{lemma}
\subsection{Setting up the Circle Method}\label{sec:CircleMethodSetup}

Fix $J\subseteq\{1,2,3,4\}$ and write $F(q):=F_{r,M,\fbm\bm{\alpha},J}(\tau)$. By Cauchy's Theorem, we have 
\begin{equation*}
c(n):=c_{r,M,\fbm\bm{\alpha},J}(n)=\frac{1}{2\pi i}\int_{\mathcal{C}}\frac{F(q)}{q^{n+1}}dq,
\end{equation*}
where $\mathcal{C}$ is an arbitrary path inside the unit circle that loops around zero in the counterclockwise direction. We choose the circle with radius $e^{-\frac{2\pi}{N^2}}$ with $N:=\lfloor\sqrt{n} \rfloor$ and the parametrization $q=e^{-\frac{2\pi}{N^2}+2\pi i t}$ with $0\leqslant t\leqslant 1$. Thus
\begin{equation*}
c(n)=\int_{0}^{1}F\left(e^{-\frac{2\pi}{N^2}+2\pi i t}\right)e^{\frac{2\pi n}{N^2}-2\pi i n t} dt.
\end{equation*}
Decomposing the path of integration along the Farey arcs $-\vartheta'_{h,k}\leqslant \Phi \leqslant \vartheta^{''}_{h,k}$ with $\Phi=t-\frac{h}{k}$,
\begin{equation}\label{eqn:c(n)}
c(n)=\sum_{\substack{0\leqslant h<k\leqslant N\\ \gcd(h,k)=1}}
e^{-\frac{2\pi i n h}{k}}\int_{-\vartheta'_{h,k}}^{\vartheta^{''}_{h,k}} F\left(e^{\frac{2\pi i}{k}(h+iz)}\right)e^{\frac{2\pi n z}{k}} d\Phi,
\end{equation}
where $z=k(\frac{1}{N^2}-i\Phi)$ as above.
Since for $J=\{1,2,3,4\}$ we may use Lemma \ref{lem:mainterm}, we consider the case that $J\neq \{1,2,3,4\}$. For $1\leq\ell \leq 4$, $\fbm\bm{ \nu}\in\N_0^4$,  $\fbm\bm{ \lambda}\in \mathcal{L}_{Mk}^4$, and $\fbm\bm{ \varepsilon}\in\{\pm\}^4$, set
\begin{align*}
\mathcal{I}_{\fbm\bm{ \nu},\fbm\bm{ \lambda},\fbm\bm{ \varepsilon},\ell}(z)&=\mathcal{I}_{\fbm\bm{ \nu},\fbm\bm{\lambda},\fbm\bm{\varepsilon},\ell,M,k}(z):=\begin{cases}
\frac{1 }{2Mk-1}&\text{if }\ell\in J,\\
\frac{\varepsilon_{\ell} }{2Mk-1}&\text{if }\ell\notin J\text{ and }\nu_{\ell}\neq 0,\\
\frac{i}{\pi} e^{\frac{\pi i r\lambda_{\ell}}{Mk}}\displaystyle{\sideset{}{^*}\sum_{\nu\geq 0}\sum_{\pm}}\mathcal{I}( \lambda_{\ell} \pm 2Mk \nu,k;z)&\text{if }\ell\notin J \text{ and }\nu_{\ell}=0,
\end{cases} \\
d_{\fbm\bm{\nu},\fbm\bm{\lambda},\fbm\bm{\varepsilon},\ell}&:=\varepsilon_{\ell} \nu_{\ell}+\delta_{\nu_{\ell}=0}\delta_{\ell\notin J} \lambda_{\ell}.
\end{align*}

\noindent
By Lemmas \ref{lem:thetatrans} and \ref{lem:Fmodular}, we have
\begin{align}\label{eqn:expandnus}
 &16M^2 F\left(e^{{\frac{2\pi i}{k}(h+iz)}}\right)\prod_{j=1}^{4}\sqrt{\alpha_j}\\
& \notag =\frac{e^{\frac{\pi z r^2}{Mk}\sum_{j=1}^4\alpha_j}}{k^2z^2}\sideset{}{^*}\sum_{\fbm\bm{\nu}\in \N_0^4}\sum_{\fbm\bm{\lambda}\in \mathcal{L}_{Mk}^4}\sum_{\fbm\bm{\varepsilon}\in\{\pm\}^4} \prod_{j=1}^{4}  e^{-\frac{\pi  \nu_j^2}{4Mk\alpha_{j} z}+\varepsilon_j\frac{\pi i r \nu_j}{Mk}}G\!\left(2M\alpha_j h,2r\alpha_jh+d_{\fbm\bm{\nu},\fbm\bm{\lambda},\fbm\bm{\varepsilon},j};k\right)\mathcal{I}_{\fbm\bm{\nu},\fbm\bm{\lambda},\fbm\bm{\varepsilon},j}(z).
\end{align}
Plugging \eqref{eqn:expandnus} back into \eqref{eqn:c(n)}, we see that the contribution to $c(n)$ from the term $\fbm\bm{\nu}\in\N_0^4$ is $\frac{1}{16M^2\prod_{j=1}^4\sqrt{\alpha_j}} \frac{1}{2^{\sum_{j=1}^4\delta_{\nu_j=0}}}$ times
\begin{align*}
\fbm\bm{I}_{\fbm\bm{\nu}}(n)=\fbm\bm{I}_{\bm{\nu},\bm{\alpha},M,J}(n):=&
\sum_{\substack{0\leq h<k\leq N \\ \gcd(h,k)=1}} \frac{e^{-\frac{2\pi i n h}{k}}}{k^2}\sum_{\fbm\bm{\lambda}\in\mathcal{L}_{Mk}^4}\sum_{\fbm\bm{\varepsilon}\in\{\pm\}^4}\prod_{j=1}^4 e^{\varepsilon_j\frac{\pi i r \nu_j}{Mk}}G\!\left(2M\alpha_j h,2r\alpha_jh+d_{\fbm\bm{\nu},\fbm\bm{\lambda},\fbm\bm{\varepsilon},j}
;k\right)
\\
&\times \int_{-\vartheta'_{h,k}}^{\vartheta''_{h,k}} \frac{1}{z^2}e^{\frac{2\pi}{k} \left(n+\frac{r^2}{2M}\sum_{j=1}^4\alpha_j\right)z -\sum_{j=1}^4\frac{\pi \nu_j^2}{4Mk\alpha_jz}}\prod_{j=1}^4
 \mathcal{I}_{\fbm\bm{\nu},\fbm\bm{\lambda},\fbm\bm{\varepsilon},j}(z)
 d\Phi.
\end{align*}

\subsection{Bounding $\sum_{\fbm\bm{\nu}\in\N_0^4\setminus\{\fbm\bm{0}\}}\fbm\bm{I}_{\bm{\nu}}(n)$}\label{sec:nointegral}
The following lemma proves useful for bounding the sum of $\fbm\bm{I}_{\bm{\nu}}(n)$ with $\fbm\bm{\nu}\neq \mathbf{{0}}$.
\begin{lemma}\label{lem:CircleNoIntegral}
Suppose that $0\leq \varrho_1\leq \varrho_2\leq \infty$, $c>0$, $r\in\Z$, and for each $0<k\leq N$ let a subset $\Lambda_k\subseteq \Z^4\setminus \{\fbm\bm{0}\}$ be given. Then, with  $\|\fbm\bm{\nu}\|^2:=\sum_{1\leq j\leq 4} \nu_j^2$,
\begin{align*}
 \sum_{1\le k\leq N} \frac{1}{k^2}
 \sum_{\fbm\bm{\nu}\in \Lambda_k}\sum_{\fbm\bm{\lambda}\in \mathcal{L}_{Mk}^4}
\prod_{j=1}^4\frac{1}{\left|\lambda_j\right|}
&\sum_{\fbm\bm{\varepsilon}\in \{\pm\}^4}
 \sum_{\varrho=\varrho_1}^{\varrho_2} \int_{\frac{1}{k(N+\varrho+1)}}^{\frac{1}{k(N+\varrho)}}\frac{1}{|z|^2}e^{-c\|\fbm\bm{\nu}\|^2 \frac{\re{\frac{1}{z}}}{k}}  d\Phi \\
&\times\left|\sum_{\substack{0\leq h<k\\ \gcd(h,k)=1\\  \varrho(h)\leq \varrho }} e^{-\frac{2\pi inh}{k}} \prod_{j=1}^4 G\!\left(2M\alpha_jh,2\alpha_j hr \pm
d_{\fbm\bm{\nu},\fbm\bm{\lambda},\fbm\bm{\varepsilon},j};k\right)\right| \ll n^{\frac{15}{16}+\varepsilon}.
\end{align*}
\end{lemma}
\begin{proof}
We first use Lemma \ref{lem:ModifiedKloostermanFund}
and the fact that
\begin{equation}\label{eqn:lambdasum}
\sum_{\lambda_j\in\mathcal{L}_{Mk}}\frac{1}{\left|\lambda_j\right|} \ll \log(k)\ll k^{\varepsilon}.
\end{equation}
Uniformly bounding against the cases $\varrho_1=0$ and $\varrho_2=\infty$ in the lemma, the left-hand side of the lemma may be bounded against
\begin{equation}\label{eqn:CircleAfterFundamental}
\ll
\sum_{1\le k\leq N} \gcd(P(n),k)^{\frac{1}{4}}k^{\frac{7}{8}+\varepsilon}
\sum_{\fbm\bm{\nu}\in \Lambda_k}
 \int_{0}^{\frac{1}{kN}}\frac{1}{|z|^2}e^{ -c\|\fbm\bm{\nu}\|^2\frac{ \re{\frac{1}{z}}}{k}}d\Phi.
\end{equation}
\noindent
By assumption, for every $\fbm\bm{\nu}\in \Lambda_k$ we have $\|\fbm\bm{\nu}\|\geq 1$ and using \eqref{realbound} we obtain that \eqref{eqn:CircleAfterFundamental} may be bounded against
\begin{equation}\label{eqn:CircleAfterFundamental2}
\ll   \sum_{1\le k\leq N} \gcd(P(n),k)^{\frac{1}{4}}k^{\frac{7}{8}+\varepsilon}
\sum_{\fbm\bm{\nu}\in \Lambda_k}e^{-\frac{c}{4}\|\fbm\bm{\nu}\|^2}
 \int_{0}^{\frac{1}{kN}}\frac{1}{|z|^2}e^{ -\frac{c}{2}\frac{ \re{\frac{1}{z}}}{k}}d\Phi.
\end{equation}
It remains to show that \eqref{eqn:CircleAfterFundamental2} is $O(n^{\frac{15}{16}+\varepsilon})$.
Since $\Lambda_k\subseteq\Z^4\setminus \{\fbm\bm{0}\}$, we may bound the sum over $\fbm\bm{\nu}$ uniformly by
\[
 \sum_{\fbm\bm{\nu}\in \Lambda_k} e^{-\frac{c}{4}\|\fbm\bm{\nu}\|^2}\leq  \sum_{\fbm\bm{\nu}\in \Z^4\setminus\{\fbm\bm{0}\}} e^{-\frac{c}{4}\|\fbm\bm{\nu}\|^2}\ll 1.
\]
We then split the sum and integral in \eqref{eqn:CircleAfterFundamental2} into three pieces:
\begin{equation*}
\sum\nolimits_1 : \quad \sum_{\leq N^{1-\ell}} \int_{0}^{\frac{1}{kN^{1+\ell}}} ,\qquad
\sum\nolimits_2 : \quad \sum_{1\le k\leq N^{1-\ell}} \int_{\frac{1}{kN^{1+\ell}}}^{
\frac{1}{kN}
} ,\qquad
\sum\nolimits_3 : \quad \sum_{ N^{1-\ell}<k\leq N} \int_{0}^{\frac{1}{kN}}
\end{equation*}
for $\ell$ some (arbitrary small) number.
We first consider $\sum\nolimits_1$. Plugging $0<|\Phi|<\frac{1}{kN^{1+\ell}}$ into the right-hand side of the equality in \eqref{realbound}, we have
\begin{equation*}
\operatorname{Re}\left(\frac{1}{z}\right)> \frac{k N^{2\ell}}{2}.
\end{equation*}
Combining this with the first inequality in \eqref{eqn:k|z|bound},
the contribution from $\sum_1$ to \eqref{eqn:CircleAfterFundamental2} is $O(e^{-\frac{c}{8}N^{2\ell}})$.

We next turn to $\sum\nolimits_2$. Using the fact $\operatorname{Re}(\frac{1}{z})>0$, we bound
\begin{equation*}
\sum\nolimits_2\ll \sum_{1\le k\leq N^{1-\ell}}  \gcd(P(n),k)^{\frac{1}{4}}k^{\frac{7}{8}+\varepsilon} \int_{\frac{1}{kN^{1+\ell}}}^{\frac{1}{kN}}\frac{1}{|z|^2}d\Phi.
\end{equation*}
One can show that the integral is $O( \frac{N^{1+\ell}}{k})$, yielding
$
\sum\nolimits_2 \ll N^{\frac{15}{8}+\frac{\ell}{8}+\varepsilon}.
$
Choosing $\ell$ sufficiently small (depending on $\varepsilon$), we obtain $\sum\nolimits_2 = O( n^{\frac{15}{16}+\varepsilon} )$. We finally turn to $\sum\nolimits_3$.  We bound, choosing $\ell \leq 8\varepsilon$
\begin{equation*}
\sum\nolimits_3 \ll N^2\sum_{ N^{1-\ell}<k\leq N}  \gcd(P(n),k)^\frac{1}{4} k^{-\frac{9}{8}+\varepsilon}
\arctan\left(\frac{N}{k}\right)\ll n^{\frac{15}{16}+\varepsilon}.  \qedhere
\end{equation*}
\end{proof}

We next bound the contribution from the sum over all $h$, $k$ of the terms $\fbm\bm{\nu}\neq \fbm\bm{0}$ from \eqref{eqn:expandnus}.
\begin{proposition}\label{prop:CircleNoIntegral}
If $J\neq \{1,2,3,4\}$, then
\[
\sideset{}{^*}\sum_{\fbm\bm{\nu}\in\N_0^4\setminus \{\fbm\bm{0}\}}\fbm\bm{I}_{\fbm\bm{\nu}}(n)=O\left(n^{\frac{15}{16}+\varepsilon}\right).
\]
\end{proposition}
\begin{proof}
Writing $k+k_j=N+\varrho_{k,j}(h)$ as in \eqref{eqn:rhohdef},
we split the integral in $\fbm\bm{I}_{\bm{\nu}}(n)$ as
\begin{equation}\label{eqn:thetasplit}
\int_{-\vartheta'_{h,k}}^{\vartheta''_{h,k}}=
\int_{-\frac{1}{k\left(N+\varrho_{k,1}(h)\right)}}^{0} +\int_{0}^{\frac{1}{k\left(N+\varrho_{k,2}(h)\right)}}
=\sum_{\varrho\ge\varrho_{k,1}(h)}
\int_{-\frac{1}{k(N+\varrho)}}^{-\frac{1}{k(N+\varrho+1)}}
+\sum_{\varrho\ge\varrho_{k,2}(h)}
\int_{\frac{1}{k(N+\varrho+1)}}^{\frac{1}{k(N+\varrho)}}.
\end{equation}

\noindent
Interchanging the sums on $h$ and $\varrho$ for the first sum in \eqref{eqn:thetasplit}, its contribution to $\fbm\bm{I}_{\fbm\bm{\nu}}(n)$ equals
\begin{multline}\label{eqn:NoIntegralFirstSum1}
\sum_{1\le k\leq N} \frac{1}{k^2}\sum_{\fbm\bm{\lambda} \in \mathcal{L}_{Mk}^4}\sum_{\fbm\bm{\varepsilon} \in\{\pm\}^4} e^{\frac{\pi i r}{Mk}\sum_{j=1}^4\varepsilon_j\nu_j} \sum_{\varrho\ge0}
\int_{-\frac{1}{k(N+\varrho)}}^{-\frac{1}{k(N+\varrho+1)}}\frac{1}{z^2} e^{\frac{2\pi}{k} \left(n+\frac{r^2}{2M}\sum_{j=1}^4\alpha_j\right)z -\sum_{j=1}^4\frac{\pi \nu_j^2}{4Mk\alpha_j z}}\\
\times \prod_{j=1}^4
\mathcal{I}_{\fbm\bm{\nu},\fbm\bm{\lambda},\fbm\bm{\varepsilon},j}(z)
d\Phi\sum_{\substack{0\leq h<k \\ \gcd(h,k)=1\\ \varrho_{k,1}(h)\leq  \varrho}} e^{-\frac{2\pi i n h}{k}}\prod_{j=1}^4
G\!\left(2M\alpha_jh,2\alpha_jhr+d_{\fbm\bm{\nu},\fbm\bm{\lambda},\fbm\bm{\varepsilon},j};k\right).
\end{multline}
Similarly, interchanging the sums over $h$ and $\varrho$ in the second sum in \eqref{eqn:thetasplit} and then applying Lemma \ref{lem:AdjacentNeighbours} yields a contribution to $\fbm\bm{I}_{\fbm\bm{\nu}}(n)$ of
\begin{multline}\label{eqn:NoIntegralSecondSum}
\sum_{1\le k\leq N} \frac{1}{k^2}\sum_{\fbm\bm{\lambda}\in \mathcal{L}_{Mk}^4}\sum_{\fbm\bm{\varepsilon}\in\{\pm\}^4} e^{\frac{\pi i r}{Mk}\sum_{j=1}^4\varepsilon_j\nu_j} \sum_{\varrho\ge0}
\int_{\frac{1}{k(N+\varrho+1)}}^{\frac{1}{k(N+\varrho)}}\frac{1}{z^2} e^{\frac{2\pi}{k} \left(n+\frac{r^2}{2M}\sum_{j=1}^4\alpha_j\right)z -\sum_{j=1}^4\frac{\pi \nu_j^2}{4Mk\alpha_j z}}\\
\times \prod_{j=1}^4
\mathcal{I}_{\fbm\bm{\nu},\fbm\bm{\lambda},\fbm\bm{\varepsilon},j}(z)
 d\Phi\sum_{\substack{0\leq h<k \\ \gcd(h,k)=1\\ \varrho_{k,1}(k-h)\leq  \varrho}} e^{-\frac{2\pi i n h}{k}}
\prod_{j=1}^4
G\!\left(2M\alpha_jh,2\alpha_jh r+d_{\fbm\bm{\nu},\fbm\bm{\lambda},\fbm\bm{\varepsilon},j}
;k\right).
\end{multline}
Making the change of variables $h\mapsto k-h$ in the inner sum, the inner sum becomes
\begin{equation*}\label{eqn:conjG}
\overline{\sum_{\substack{0\leq h<k \\ \gcd(h,k)=1\\ \varrho_{k,1}(h)\leq  \varrho}} e^{-\frac{2\pi i n h}{k}}\prod_{j=1}^{4}G\!\left(2M\alpha_jh,2\alpha_jhr -d_{\fbm\bm{\nu},\fbm\bm{\lambda},\fbm\bm{\varepsilon},j}
;k\right)}.
\end{equation*}
We then take the absolute value inside all of the sums except the sum on $h$ in both \eqref{eqn:NoIntegralFirstSum1} and \eqref{eqn:NoIntegralSecondSum}. Noting that $|z|^2$ and $\operatorname{Re}(\frac{1}{z})$ are the same for $\Phi$ and $-\Phi$, we may make the change of variables $\Phi\mapsto -\Phi$ in \eqref{eqn:NoIntegralFirstSum1} to bound both \eqref{eqn:NoIntegralFirstSum1} and \eqref{eqn:NoIntegralSecondSum} against
\begin{multline}\label{eqn:NoIntegralFirstSum}
\ll \sum_{1\le k\leq N}\frac{1}{k^2} \sum_{\fbm\bm{\lambda}\in \mathcal{L}_{Mk}^4}\sum_{\fbm\bm{\varepsilon}\in \{\pm\}^4}
 \sum_{\varrho\ge0}
\int_{\frac{1}{k(N+\varrho+1)}}^{\frac{1}{k(N+\varrho)}}
|z|^{-2} e^{\frac{2\pi}{k} \left(n+\frac{r^2}{2M}\sum_{j=1}^4\alpha_j\right)\re{z} -\sum_{j=1}^4\frac{\pi \nu_j^2}{4Mk\alpha_j}\re{\frac{1}{z}}} \\
\times \left|\mathcal{I}_{\fbm\bm{\nu},\fbm\bm{\lambda},\fbm\bm{\varepsilon},j}(z)\right|
d\Phi
\left|\sum_{\substack{0\leq h<k \\ \gcd(h,k)=1\\ \varrho_{k,1}(h)\leq  \varrho}} e^{-\frac{2\pi i n h}{k}}G\!\left(2M\alpha_jh,2\alpha_jhr \pm d_{\fbm\bm{\nu},\fbm\bm{\lambda},\fbm\bm{\varepsilon},j};k\right)\right|,
\end{multline}
\noindent where $\pm$ is chosen as ``$+$'' for \eqref{eqn:NoIntegralFirstSum1} and ``$-$'' for \eqref{eqn:NoIntegralSecondSum}. Note that since $\operatorname{Re}(z)=\frac{k}{N^2}\sim \frac{k}{n}$, we have
\begin{equation}\label{eqn:trivialexponential}
e^{\frac{2\pi}{k}\left(n+\frac{r^2}{2M}\sum_{j=1}^{4} \alpha_j\right)\re{z}}\ll  1.
\end{equation}
We next bound $|\mathcal{I}_{\fbm\bm{\nu},\fbm\bm{\lambda},\fbm\bm{\varepsilon},j}(z)|$. In the case that $j\in J$ or $\nu_j\neq 0$, we trivially bound (using $\lambda_j\in\mathcal{L}_{Mk}$)
\[
\left|\mathcal{I}_{\fbm\bm{\nu},\fbm\bm{\lambda},\fbm\bm{\varepsilon},j}(z)\right|=\frac{1}{2Mk-1}<\frac{1}{|\lambda_j|}.
\]
If both $j\notin J$ and $\nu_j=0$, then we use \eqref{eqn:intsumboundO3} to bound
\begin{equation}\label{eqn:Itrivial}
\left|\mathcal{I}_{\fbm\bm{\nu},\fbm\bm{\lambda},\fbm\bm{\varepsilon},j}(z)\right|
\ll \frac{n^{\varepsilon}}{|\lambda_j|}.
\end{equation}
 Hence, setting
 $c:=\frac{\pi}{4M\max_{j}(|\alpha_j|)}$, \eqref{eqn:NoIntegralFirstSum} may be bounded against
\begin{multline*}
\ll  n^{\varepsilon}
\sum_{1\le k\leq N}\frac{1}{k^2}\sum_{\fbm\bm{\lambda}\in \mathcal{L}_{Mk}^4}\prod_{j=1}^{4}\frac{1}{|\lambda_{j}|}\sum_{\fbm\bm{\varepsilon}\in \{\pm\}^4}
 \sum_{\varrho\ge0}\int_{\frac{1}{k(N+\varrho+1)}}^{\frac{1}{k(N+\varrho)}}\frac{1}{|z|^2} e^{-c \|\fbm\bm{\nu}\|^2\frac{\re{\frac{1}{z}}}{k}} d\Phi\\
\times \left|\sum_{\substack{0\leq h<k\leq N \\ \gcd(h,k)=1\\ \varrho_{k,1}(h)\leq  \varrho}} e^{-\frac{2\pi i n h}{k}}
\prod_{j=1}^{4}
G\!\left(2M\alpha_jh,2\alpha_jh r\pm
d_{\fbm\bm{\nu},\fbm\bm{\lambda},\fbm\bm{\varepsilon},j}
;k\right)\right|.
\end{multline*}
By Lemma \ref{lem:rhoh}, we have $\varrho(h)=\varrho_{k,1}(h)$.
Summing over $\fbm\bm{\nu}\in \N_0^4\setminus\{\fbm\bm{0}\}$, we may therefore use Lemma \ref{lem:CircleNoIntegral} with $\Lambda_k=\N_0^4\setminus\{\fbm\bm{0}\}$, $\varrho_1=0$, and $\varrho_2=\infty$ to conclude that $\sideset{}{^*_{\bm{\nu}\neq 0}}\sum \fbm\bm{I}_{\fbm\bm{\nu}}(n)$ is $O(n^{\frac{15}{16}+\varepsilon})$, giving the bound claimed in the proposition.
\end{proof}

\subsection{Bounding $\fbm\bm{I}_{\bm{0}}(n)$}

This subsection is devoted to bounding $\fbm\bm{I}_{\bm{0}}(n)$.
\begin{proposition}\label{prop:CircleConstantTerms}
If $J\neq \{1,2,3,4\}$, then
\[
\fbm\bm{I}_{\bm{0}}(n)=O\left(n^{\frac{15}{16}+\varepsilon}\right).
\]
\end{proposition}
\begin{proof}
As in the proof of Proposition \ref{prop:CircleNoIntegral}, we first split the sum as in \eqref{eqn:thetasplit} and interchange the sums on $h$ and $\varrho$ and then take the absolute value inside all of the sums other than the sum on $h$. Since $J\neq \{1,2,3,4\}$, without loss of generality we have $4\notin J$. For $1\leq j\leq 3$, we use \eqref{eqn:Itrivial} and we bound $\mathcal{I}_{\fbm\bm{0},\fbm\bm{\lambda},\fbm\bm{\varepsilon},4}(z)$ with \eqref{eqn:intsumboundO2}.
Plugging in \eqref{eqn:trivialexponential}, we hence obtain
\begin{multline}
\bm{I}_{\bm{0}}(n)\ll  n^{\varepsilon}\sum_{1\le k\leq N} \frac{1}{k^2} \sum_{\varrho\ge0}\int_{\frac{1}{k(N+\varrho+1)}}^{\frac{1}{k(N+\varrho)}}\!\frac{1}{|z|^2}%\\\times
\sum_{\fbm\bm{\lambda}\in \mathcal{L}_{Mk}^4}\!O\left(\frac{\sqrt{k|z|}}{|\lambda_4|}+\left(1+\left|\log(k|z|)\right|\right)e^{-\frac{c\lambda_4^2\re{\frac{1}{z}}}{k}}\right)d \Phi\\
\times
\prod_{j=1}^3\frac{1}{|\lambda_j|}
 \left| \sum_{\substack{0\leq h<k \\ \text{gcd}(h,k)=1\\
\varrho_{k,1}(h)\leq \varrho}} e^{-\frac{2\pi inh}{k}} \prod_{j=1}^{4} G\!\left(2M\alpha_jh, 2\alpha_j hr\pm \delta_{j\notin J}\lambda_{j}
;k\right) \right|. \label{eqn:oneintegral2}
\end{multline}
Plugging in Lemma \ref{lem:rhoh}, the contribution to $\bm{I}_{\bm{0}}(n)$ from the first term in the $O$-constant in \eqref{eqn:oneintegral2} is bounded by
\[
\ll\sum_{1\le k\leq N} \frac{n^{\varepsilon}}{k^{\frac{3}{2}}} \sum_{\fbm\bm{\lambda}\in\mathcal{L}_{Mk}^4}\prod_{j=1}^4\frac{1}{|\lambda_j|}
\sum_{\varrho\ge0} \int_{\frac{1}{k(N+\varrho+1)}}^{\frac{1}{k(N+\varrho)}}
\hspace{-2pt} \frac{d\Phi}{|z|^{\frac{3}{2}}}
%\times
\left|\sum_{\substack{0\leq h<k \\ \gcd(h,k)=1\\
\varrho(h)\leq \varrho}} \hspace{-4pt} e^{-\frac{2\pi inh}{k}} \prod_{j=1}^{4} G\!\left(2M\alpha_jh, 2\alpha_j hr\pm\delta_{j\notin J}\lambda_j;k\right)\right|.
\]
Using Lemma \ref{lem:ModifiedKloostermanFund} and \eqref{eqn:lambdasum}, we can bound this against
\begin{align}\label{eqn:mainintbound}
\nonumber &\ll n^{\varepsilon} \sum_{1\le k\leq N} \sqrt{\gcd(P(n),k)}k^{\frac{11}{8}+\varepsilon}
\sum_{\fbm\bm{\lambda}\in\mathcal{L}_{Mk}^4}\prod_{j=1}^4\frac{1}{|\lambda_j|}
\sum_{\varrho\ge0}
\int_{\frac{1}{k(N+\varrho+1)}}^{\frac{1}{k(N+\varrho)}}
\frac{d\Phi}{|z|^{\frac{3}{2}}} \\
&\ll n^{\varepsilon} \sum_{1\le k\leq N} \sqrt{\gcd(P(n),k)} k^{\frac{11}{8}+\varepsilon}
\int_{0}^{\frac{1}{kN}}
\frac{d\Phi}{|z|^{\frac{3}{2}}}.
\end{align}

We split the integral in \eqref{eqn:mainintbound} into the ranges $\Phi<\frac{1}{N^2}$ and $\Phi\geq \frac{1}{N^2}$. Using that for  $\Phi\geq  \frac{1}{N^2}$, we have $|z|^{\frac{3}{2}}\gg k^{\frac{3}{2}}\Phi^{\frac{3}{2}}$, the contribution from $\Phi\geq \frac{1}{N^2}$ to \eqref{eqn:mainintbound} may be bounded against
\[
\ll n^{\varepsilon}\sum_{1\le k\leq N} \sqrt{\gcd(P(n),k)} k^{-\frac{1}{8}+\varepsilon}
\int_{\frac{1}{N^2}}^{\infty}
 \Phi^{-\frac{3}{2}}d\Phi
\ll  n^{\varepsilon} N\sum_{0<k<N} \sqrt{\gcd(P(n),k)} k^{-\frac{1}{8}+\varepsilon}.
\]
For $0<\Phi<\frac{1}{N^2}$, we use the trivial bound $|z|^{\frac{3}{2}}\gg \frac{k^{\frac{3}{2}}}{N^3}$,
to obtain that the contribution from $0<\Phi<\frac{1}{N^2}$ to \eqref{eqn:mainintbound} is
\[
\ll  n^{\varepsilon} N\sum_{1\le k\leq N} \sqrt{\gcd(P(n),k)} k^{-\frac{1}{8}+\varepsilon}.
\]
Therefore \eqref{eqn:mainintbound} is $O(n^{\frac{15}{16}+\varepsilon})$.

We next consider the contribution to \eqref{eqn:oneintegral2} coming from the second $O$-term. Using \eqref{eqn:k|z|bound} to bound $1+|\log(k|z|)|\ll n^{\varepsilon}$, the contribution to \eqref{eqn:oneintegral2} from the second term in the $O$-constant is
\begin{multline*}
\ll  n^{\varepsilon}\sum_{1\le k\leq N} k^{-2}
\sum_{\fbm\bm{\lambda}\in\mathcal{L}_{Mk}^4}\prod_{j=1}^3\frac{1}{|\lambda_j|}\sum_{\fbm\bm{\varepsilon}\in \{\pm\}^4}
\sum_{\varrho\ge0}
\int_{\frac{1}{k(N+\varrho+1)}}^{\frac{1}{k(N+\varrho)}}
 \frac{1}{|z|^2} e^{-\frac{c \lambda_4^2\re{\frac{1}{z}}}{k}} d\Phi\\
 \times\left| \sum_{\substack{0\leq h<k \\ \text{gcd}(h,k)=1\\ \varrho_{k,1}(h)\leq \varrho}} e^{-\frac{2\pi inh}{k}} \prod_{j=1}^{4} G\!\left(2M\alpha_jh, 2\alpha_j hr+\delta_{j\notin J}\lambda_j;k\right) \right| \ll n^{\frac{15}{16} + \varepsilon},
\end{multline*}
employing Lemma \ref{lem:CircleNoIntegral} with $\varrho_1=0$, $\varrho_2=\infty$, and $
\Lambda_k=\{(0\ 0\ 0\ \lambda_4)^T: \lambda_4\in \mathcal{L}_{Mk}\}$ yields that this may be bounded against $O(n^{\frac{15}{16}+\varepsilon}).$
\end{proof}

\subsection{Proof of Theorem \ref{thm:rNrZ}}
We are now ready to prove the main theorem.
\begin{proof}[Proof of Theorem \ref{thm:rNrZ}]
(1)
We first use Lemma \ref{lem:ThetaFalse}. If $M$ is odd, then we use the fact that
\begin{equation*}\label{eqn:ThetaModd}
\Theta_{r,M,\fbm\bm{\alpha}}^+(\tau)=\Theta_{2r,2M,\fbm\bm{\alpha}}^+\left(\frac{\tau}{4}\right).
\end{equation*}
 Thus we may assume without loss of generality that $M$ is even. We deal with the terms from Lemma \ref{lem:ThetaFalse} termwise for each $J\subseteq\{1,2,3,4\}$.

Plugging Propositions \ref{prop:CircleNoIntegral} and \ref{prop:CircleConstantTerms} into \eqref{eqn:expandnus}, we conclude that for $J\neq \{1,2,3,4\}$ we have
\[
c_{r,M,\fbm\bm{\alpha},J}(n)=O\left(n^{\frac{15}{16}+\varepsilon}\right).
\]
Thus by Lemma \ref{lem:ThetaFalse}, we have
\[
s_{r,2M,\fbm\bm{\alpha}}\left(2Mn+r^2\sum_{1\leq j \leq 4}\alpha_j\right)= \frac{1}{16}c_{r,M,\fbm\bm{\alpha}}(n)+ O\left(n^{\frac{15}{16}+\varepsilon}\right).
\]
Plugging in Lemma \ref{lem:mainterm} (1) then yields
\[
s_{r,2M,\fbm\bm{\alpha}}\left(2Mn+r^2\sum_{1\leq j \leq 4}\alpha_j\right)= \frac{1}{16}s_{r,2M,\fbm\bm{\alpha}}^*\left(2Mn+r^2\sum_{1 \leq j \leq 4}\alpha_j\right)+ O\left(n^{\frac{15}{16}+\varepsilon}\right).
\]
Since
\[
s_{r,2M,\fbm\bm{\alpha}}(n)=s_{r,2M,\fbm\bm{\alpha}}^*(n)=0
\]
if $n\not\equiv r^2\sum_{j=1}^4\alpha_j\pmod{2M}$, the claim follows.

\noindent (2) By Lemma \ref{lem:rexactly}, we have
\[
r_{m,\fbm\bm{\alpha}}(n)=r_{m,\fbm\bm{\alpha}}^+(n)+O\left(n^{\frac{1}{2}+\varepsilon}\right).
\]
Lemma \ref{lem:r+s+rel} then yields
\[
r_{m,\fbm\bm{\alpha}}^+(n)=s_{m,2(m-2),\fbm\bm{\alpha}}\left(8(m-2)\left(n-\sum_{1\leq j \leq 4}\alpha_j\right)+m^2\sum_{1 \leq j \leq 4}\alpha_j\right).
\]
Thus by part (1) and Lemma \ref{lem:mainterm} we have
\begin{align*}
r_{m,\fbm\bm{\alpha}}^+(n)&=\frac{1}{16}s_{m,2(m-2),\fbm\bm{\alpha}}^*\left(8(m-2)\left(n-\sum_{1\leq j\leq 4}\alpha_j\right)+m^2\sum_{1\leq j \leq 4}\alpha_j\right)+O\left(n^{\frac{15}{16}+\varepsilon}\right)\\
&=\frac{1}{16}r_{m,\fbm\bm{\alpha}}^*(n)+ O\left(n^{\frac{15}{16}+\varepsilon}\right). \qedhere
\end{align*}
\end{proof}
\section{Proof of Corollary \ref{cor:hexagonal} and Corollary \ref{cor:hexagonal2}}\label{sec:Corollaries}
In this section, we prove Corollaries \ref{cor:hexagonal} and \ref{cor:hexagonal2}.

\begin{proof}[Proof of Corollary \ref{cor:hexagonal}]
By Theorem \ref{thm:rNrZ} (2), we have
\begin{equation}\label{eqn:hexagonal1}
r_{6,\fbm\bm{\alpha}}(n)=\frac{1}{16} r_{6,\fbm\bm{\alpha}}^*(n)+O\left(n^{\frac{15}{16}+\varepsilon}\right).
\end{equation}
Completing the square in the special case $\fbm\bm{\alpha}=(1,1,1,1)$, we obtain
\[
r_{6,(1,1,1,1)}^*(n)= s_{3,4,(1,1,1,1)}^*(8n+4).
\]
Note that by the change of variables $x_j\mapsto \varepsilon_j x_j$ with $\bm{\varepsilon}\in \{\pm\}^4$, we have
\[
s_{3,4,(1,1,1,1)}^*(8n+4)=\frac{1}{16}s_{1,2,(1,1,1,1)}^*(8n+4).
\]
Cho \cite[Example 3.3]{Cho} computed
\[
s_{1,2,(1,1,1,1)}^*(8n+4)=16 \sigma(2n+1).
\]
Thus
\[
r_{6,(1,1,1,1)}^*(n)= \sigma(2n+1).
\]
Plugging this back into \eqref{eqn:hexagonal1} yields the claim.
\end{proof}

We next prove Corollary \ref{cor:hexagonal2}.
\begin{proof}[Proof of Corollary \ref{cor:hexagonal2}]
Using \cite[Example 3.4]{Cho}, the argument is essentially identical to the proof of Corollary \ref{cor:hexagonal}, except that in this case it is not immediately obvious that the main term is always positive. For this we use multiplicativity to bound, with $\varphi$ denoting Euler's totient function,
\[
-\sum_{d\mid (8n+5)} \left(\frac{8}{d}\right) d\geq \varphi(8n+5)\gg n^{1-\varepsilon}.\qedhere
\]
\end{proof}

We finally prove Corollary \ref{cor:pentagonal}.
\begin{proof}[Proof of Corollary \ref{cor:pentagonal}]
By Theorem \ref{thm:rNrZ} (2), we have
\begin{equation}\label{eqn:pentagonal}
r_{5,(1,1,1,1)}(n)=\frac{1}{16} r_{5,(1,1,1,1)}^*(n)+O\left(n^{\frac{15}{16}+\varepsilon}\right).
\end{equation}
Completing the square, we obtain
\begin{equation}\label{eqn:rspentagonal}
r_{5,(1,1,1,1)}^*(n)= s_{5,6,(1,1,1,1)}^*(24n+4).
\end{equation}
Using \cite[Proposition 2.1]{Shimura}, it is not hard to show that the generating function $\Theta_{5,6,(1,1,1,1)}^*$ for $s_{5,6,(1,1,1,1)}^*$ is a modular form of weight two on $\Gamma_0(144)$.
We next claim that
\begin{equation}\label{eqn:Thetamin16split}
\Theta_{5,6,(1,1,1,1)}^*(\tau)= \frac{2}{3}E(4\tau)+\frac{1}{3}\eta^4(24\tau),
\end{equation}
where
\begin{align*}
E(\tau):=\sum_{n\equiv 1\pmod{6}} \sigma(n) q^n, \quad \eta(\tau):=q^{\frac{1}{24}}\prod_{n\geq 1}\left(1-q^n\right).
\end{align*}
For this we note first that $\tau\mapsto\eta^4(24\tau)$ is a cusp form of weight two on $\Gamma_0(144)$. We next recall that for a translation-invariant function $f$ with Fourier expansion
$
f(\tau)=\sum_{n\geq 0} c_f(v;n) q^n,
$
the \emph{quadratic twist of $f$ with a character $\chi$} is given by
\[
f\otimes \chi(\tau):=\sum_{n\geq 0} \chi(n) c_{f}(v;n) q^n.
\]
For $\delta\in\N$, one also defines the \emph{$V$-operator} and \emph{$U$-operator} by
\begin{equation*}
f\big| V_{\delta}(\tau):=\sum_{n\geq 0} c_{f}\left(\delta v;n\right)q^{\delta n},\quad
f\big| U_{\delta}(\tau):=\sum_{n\geq 0} c_{f}\left(\frac{v}{\delta};\delta n\right)q^{n}.
\end{equation*}
A straightforward generalization of the proof for holomorphic modular forms (see \cite[Proposition 17 (b) of Section 3]{Koblitz} and \cite[Lemma 1]{LiWinnie}) yields that if $f$ satisfies weight $k\in\Z$ modularity on $\Gamma_0(N)$ and $\chi$ is a character with modulus $M$, then $f\otimes \chi$ satisfies weight $k$ modularity on $\Gamma_0(\operatorname{lcm}(N,M^2))$ with character $\chi^2$, $f|U_{\delta}$ satisfies weight $k$ modularity on
$\Gamma_0(\operatorname{lcm}(\frac{N}{\gcd(N,\delta)},\delta))$, and $f|V_{\delta}$ satisfies weight $k$ modularity on $\Gamma_0(\delta N)$.
Recall the weight two Eisenstein series
\[
E_2(\tau):=1-24\sum_{n\geq 1} \sigma(n) q^n
\]
and
its modular completion $\widehat{E}_2(\tau):=E_2(\tau)-\frac{3}{\pi v}$,
 which is well-known to be modular of weight two on $\SL_2(\Z)$.
 Setting $\chi_D(n):=(\frac{D}{n})$, we
 see that
\[
E=-\frac{1}{48}\left(E_2\otimes \chi_{-3}+E_2\otimes \chi_{-3}^2\right)\big|\left(1-U_2\circ V_2\right).
\]
Since the constant term is annihilated by $\chi_{-3}$, we furthermore have
\[
E=-\frac{1}{48}\left(\widehat{E}_2\otimes \chi_{-3}+\widehat{E}_2\otimes\chi_{-3}^2\right)\big|\left(1-U_2\circ V_2\right),
\]
and hence $E$ is modular of weight two on $\Gamma_0(36)$. Since it is holomorphic, we conclude that the right-hand side of \eqref{eqn:Thetamin16split} is a weight two modular form on $\Gamma_0(144)$. By the valence formula, \eqref{eqn:Thetamin16split} is true as long as it is true for the first 48 Fourier coefficients, which is easily checked with a computer.

By work of Deligne \cite{Deligne}, we know that the $n$-th Fourier coefficient of $\eta^4(24\tau)$ is $\ll n^{\frac{1}{2}+\varepsilon}$. Therefore, writing the $n$-th Fourier coefficient of $E$ as $c_E(n)$, we conclude from \eqref{eqn:Thetamin16split} that
\[
s_{5,6,(1,1,1,1)}^*(24n+4)=\frac{2}{3}c_E(6n+1)+O\left(n^{\frac{1}{2}+\varepsilon}\right)= \frac{2}{3}\sigma(6n+1) + O\left(n^{\frac{1}{2}+\varepsilon}\right).
\]
Plugging into \eqref{eqn:rspentagonal} and then plugging this into \eqref{eqn:pentagonal} implies that
\[
r_{5,(1,1,1,1)}(n)=\frac{1}{16}r_{5,(1,1,1,1)}^*(n)+O\left(n^{\frac{15}{16}+\varepsilon}\right) = \frac{1}{24}\sigma(6n+1) + O\left(n^{\frac{15}{16}+\varepsilon}\right).\qedhere
\]
\end{proof}
\section*{Data availability and conflict of interest statements}
Data sharing is not applicable to this article as no datasets were generated or analyzed during the current study. The authors also declare that there are no conflicts of interest.

\end{document}